\documentclass[reqno]{amsart}

\usepackage{a4wide}
\usepackage[english]{babel} 

\usepackage{amsmath,amssymb,amsfonts,amsthm, mathrsfs, esint, dsfont}
\usepackage{moreverb,rotating,graphics, tikz, caption, subcaption, float}
\usepackage{framed,fancybox}
\usepackage{enumerate, enumitem}
\usepackage{bbm, color, csquotes}

\usepackage[colorlinks, linkcolor={blue}, citecolor={red}, urlcolor = {blue}]{hyperref}

\usepackage[backend=biber,
style=alphabetic,
maxbibnames=99,
doi=false, url=false, giveninits=true, date=year]{biblatex}

\AtEveryBibitem{
    \clearfield{issn}
}    
\newtheorem{theorem}{Theorem}[section]
\newtheorem{proposition}[theorem]{Proposition}
\newtheorem{remark}[theorem]{Remark}
\newtheorem{lemma}[theorem]{Lemma}
\newtheorem{conjecture}[theorem]{Conjecture}
\newtheorem{corollary}[theorem]{Corollary}

\theoremstyle{definition}

\newcommand{\RR}{\mathbb{R}}

\newcommand{\ri}{\mathrm{i}}
\newcommand{\cS}{\mathscr{S}}

\newcommand{\cE}{\mathcal{E}}
\newcommand{\cL}{\mathcal{L}}

\newcommand{\cR}{\mathcal{R}}

\newcommand{\R}{\mathbb{R}}
\newcommand{\N}{\mathbb{N}}

\newcommand\be{{\mathbf e}}

\newcommand\bk{{\mathbf k}}

\newcommand\bx{{\mathbf x}}
\newcommand\by{{\mathbf y}}

\newcommand\bnull{{\mathbf 0}}

\def\cB{{\mathcal B}}

\def\cE{{\mathcal E}}
\def\cF{{\mathcal F}}

\def\cI{{\mathcal I}}

\def\cL{{\mathcal L}}
\def\cM{{\mathcal M}}

\def\cP{{\mathcal P}}

\def\cR{{\mathcal R}}
\def\cS{{\mathcal S}}

\def\fa{{\mathfrak{a}}}

\def\rd{{\mathrm{d}}}
\def\re{{\mathrm{e}}}
\def\ri{{\mathrm{i}}}

\newcommand\1{{\ensuremath {\mathds 1} }} 

\newcommand{\norm}[1]{\left\| #1 \right\|}

\newcommand{\myfootnote}[1]{
    \renewcommand{\thefootnote}{}
    \footnotetext{\scriptsize#1}
    \renewcommand{\thefootnote}{\arabic{footnote}}
}

\def\NN{{\mathbb N}}

\def\RR{{\mathbb R}}

\newcommand\ii{{\infty}}

\newcommand{\set}[1]{\left\{ #1\right\}}    
\newcommand{\bra}[1]{\left( #1\right)}
\newcommand{\av}[1]{\left| #1\right|}

\def\fa{{\mathfrak a}}
\def\fb{{\mathfrak b}}
\def\fc{{\mathfrak c}}
\newcommand{\Tr}{{\rm Tr}}
\newcommand{\extreme}{{\rm Extreme}}

\title[]{Existence and non-existence of minimizers for a multi--particle model with concave nonlinearity}

\author{David Gontier \qquad Salma Lahbabi$^*$ \qquad Simona Rota Nodari}

\date{\today}

\begin{document}

    \myfootnote{David Gontier: CEREMADE, Université Paris-Dauphine, PSL University, 75016 Paris, France \& ENS/PSL University, DMA, F-75005, Paris, France;\\
        email: \href{gontier@ceremade.dauphine.fr}{gontier@ceremade.dauphine.fr}}
    \myfootnote{Salma Lahbabi: Equipe de Mathématiques Appliquées, LARILE, ENSEM, Université Hassan II de Casablanca, Casablanca, Morocco;\\
        email: \href{s.lahbabi@ensem.ac.ma}{s.lahbabi@ensem.ac.ma}}
     \myfootnote{Simona Rota Nodari: Laboratoire J.A. Dieudonné, Université Côte d'Azur, CNRS UMR  7351, 06108 Nice, France \& Institut Universitaire de France;\\
        email: \href{simona.rotanodari@univ-cotedazur.fr}{simona.rotanodari@univ-cotedazur.fr}}

    \begin{abstract}
        We study a multi--particle model including a kinetic energy and a non linear local self-interaction, both in the bosonic and fermionic cases. In both cases, we prove that the model is well-posed if the number of particles is large enough. In particular, we show that there is a nonlinearity for which the model with $N=2$ particles is well-posed, while the model with $N=1$ is not.
        
        \bigskip
        \noindent \sl \copyright~2025 by the authors. This paper may be reproduced, in its entirety, for non-commercial purposes.
    \end{abstract}
    
    \maketitle    
    
    \tableofcontents


\section{Introduction}

\subsection{Presentation of the model and main results}
The main goal of this paper is to exhibit a model for which the problem with $N = 2$ particles is well-posed, while the problem with $N = 1$ is not. In some sense, the model with $N = 2$ particles is stable only because of the presence of the nonlinear effects between the two particles. In the bosonic case, such a phenomenon has been observed in~\cite{LewRot-20}, where it is shown that the corresponding problem is well--posed only when the total mass of the particles is large enough. We are able to adapt the results in the fermionic case, using tools from~\cite{GonLewNaz-21}. In order to state our results, let us describe our model.

We study a non--linear model for multi-particle systems, in a bosonic regime and in an uncorrelated fermionic regime. In this setting, a system of $\lambda \ge 0$ particles is described by a positive one--body density matrix $\gamma \in \cS \left( L^2(\R^d) \right)$, where $\cS \left( L^2(\R^d) \right)$ is the set of self-adjoint bounded operators acting on $L^2(\R^d)$, satisfying $\gamma \ge 0$ and $\Tr(\gamma)  = \lambda$. Usually, $\lambda$ is an integer ($\lambda = N$ for a system with $N$ particles), but it is interesting to study this problem for any $\lambda \ge 0$. For such an operator, we denote by $\rho_\gamma$ its density. In the fermionic case, we have the additional Pauli principle $0 \le \gamma \le 1$.

\medskip

For such a system described by $\gamma \in \cS \left( L^2(\R^d) \right)$, we consider the energy
\begin{equation} \label{eq:def:Ealpha}
    \boxed{ \cE_\alpha(\gamma) := \Tr\bra{-\Delta \gamma}+ \int_{\R^d} F_\alpha(\rho_\gamma).}
\end{equation}
The first term represents the kinetic energy of the system, while the second term is a non--linear local self-interaction. In this work, we focus on the  case where $F_\alpha$ is of the form
 \begin{equation} \label{eq:scaling_Falpha}
    F_\alpha(t) := \alpha^{q} F_1\left( \frac{t}{\alpha} \right),
    \quad \text{with} \quad
     F_1(t) := \begin{cases}
        - t^q & \quad t \le 1 \\
        - \fa + \fb t - \fc t^{r} & \quad t \ge 1,
    \end{cases}
\end{equation}
where the parameters $r$ and $q$, and the constants $\fa$, $\fb$ and $\fc$ satisfy
\[
    \boxed{ 1 < r < \frac{d+2}{d} < q, } \qquad \text{and} \qquad 
    \fa := \frac{(q-1)(q-r)}{r} , \quad \fb := \frac{q(q-r)}{r-1}, \quad \fc := \frac{q(q-1)}{r(r-1)}.
\]
These parameters are chosen so that the function $F_1$ is concave, of class $C^2$, is supercritical as $t \to 0$ and subcritical as $t \to \infty$ (see Figure~\ref{fig:F1}).
\begin{figure}[!h]
	\includegraphics[scale=0.4]{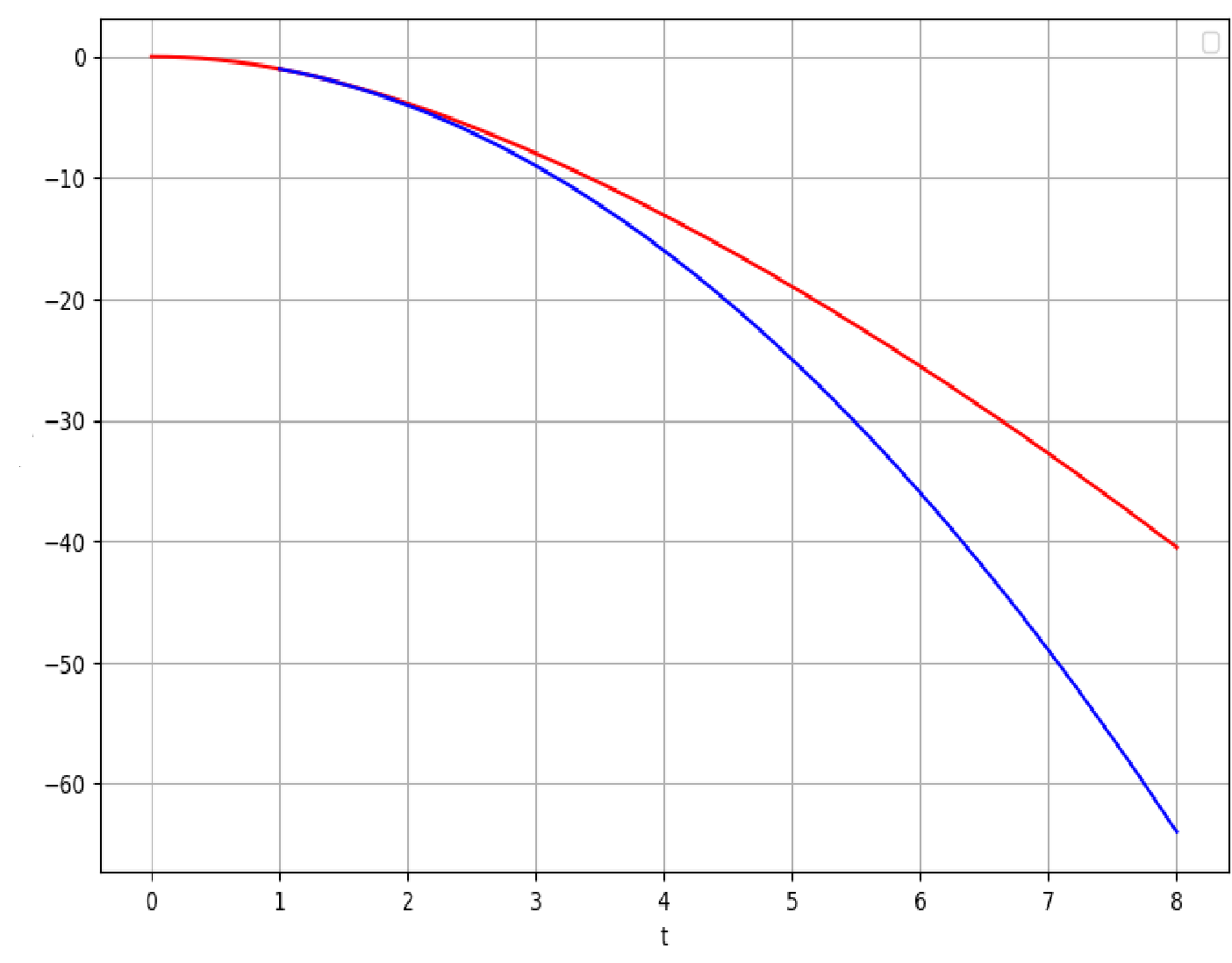}
	\caption{The function $F_1$ in red for the choice of parameters $d=3, r=4/3$ and $q=2$. In blue, the function $-t^q$.}
		\label{fig:F1}
\end{figure}
 Our scaling in~\eqref{eq:scaling_Falpha} is chosen so that $F_\alpha(t) = -t^q$ on $[0, \alpha]$. One should therefore think of the function $F_\alpha$ as a concave regularization of the function $-t^q$ (in the sense that $F_\alpha$ decays slower than $-t^q$ at infinity, since $r < q$), and where the regularization happens at $t = \alpha$.

\medskip

We consider the minimization problems
\begin{align*}
    J_\alpha(\lambda) & := \inf\set{  \cE_\alpha(\gamma), \quad  \gamma \in \cS \left( L^2(\R^d) \right),  \quad  \Tr(\gamma)=\lambda }, \quad \text{(bosonic case)}\\
    I_\alpha(\lambda) & :=  \inf\set{\cE_\alpha(\gamma), \quad \gamma \in \cS \left( L^2(\R^d) \right), \quad 0 \le \gamma \le 1, \quad  \Tr(\gamma)=\lambda}, \quad  \text{(fermionic case)}.
\end{align*}

\begin{theorem}[Bosonic case]
    \label{th:main_bosonic_intro}
Let $d\geq 2$.    For all $\alpha > 0$, there is $\lambda_c(\alpha) > 0$ so that for all $\lambda \ge \lambda_c(\lambda)$, the bosonic  problem has a minimizer, while for all $0 < \lambda < \lambda_c(\alpha)$, it does not.
\end{theorem}

\begin{theorem}[Fermionic case]
    \label{th:main_fermionic_intro}
    Let $d \ge 3$ and $q < 2$. Then there exists $\alpha > 0$ such that, the fermionic problem with $\lambda = 1$ has no minimizers, while the one with $\lambda = 2$ is well-posed.
\end{theorem}

The first result states that the bosonic problem is well--posed only for $\lambda$ larger than some critical mass $\lambda_c > 0$. As noticed in~\cite{LewRot-20}, this comes from the fact that $F_\alpha$ is supercritical as $t \to 0$. The second result states that, in the fermionic case, the $\lambda = 2$ particle system is stable, while the $\lambda = 1$ is not. In some sense, the two particles stabilize themselves thanks to the (small) tunnelling effect between them. Note that the two conditions on $q$, namely that $\frac{d+2}{d} < q < 2$ can only hold in dimensions $d \ge 3$.


\subsection{What is know for non-linear Schrödinger equations}

Let us first make some comments. In the bosonic case, it is not difficult to see that one can restrict the minimization to rank--one operators of the form $\gamma = | u \rangle \langle u |$, for some (not normalized) function $u$ satisfying $\| u \|_2^2 = \lambda$ (see Section~\ref{sec:bosons_Jalpha}). In the case $F_\infty(t) := - t^q$, which corresponds somehow to the case $\alpha \to \infty$, we obtain the equivalent minimization problem
\[
J_\infty(\lambda):= \inf\set{ \int_{\R^d} | \nabla u |^2 - \int_{\R^d} | u |^{2q}, \quad u \in H^1(\R^d), \quad u \ge 0, \quad \quad  \int_{\R^d} | u |^2 =\lambda},
\] 
which is associated to the non--linear Schrödinger (NLS) equation
\[
    -  \Delta u - q u^{2q-1} + \mu u = 0, \qquad \mu \ge 0.
\]
This model has been extensively studied in the literature, and we recall here the main results for this problem. First, we have that for all $q < \frac{d}{d-2}$ ($q < \infty$ if $d = 1,2$) and all $\mu > 0$, the NLS equation always admits a unique positive solution $u_\mu$, up to translation, and this solution is radial decreasing, see~\cite{Cof-72, Kwo-89, McL-93}. Actually,  these solutions for different $\lambda$ are linked by a simple scaling relation, namely 
\begin{equation*} 
    u_\lambda(\bx) =\lambda^a u_1(\lambda^b\bx),
\end{equation*} 
with $a=d^{-1}\bra{\frac{d+2}{d}-q}^{-1}$ and $b=(q-1)a$, so that this solution has an $L^2$ mass
\[
    \lambda(\mu) := \| u_\mu \|_2^2 = \lambda(1) \mu^{\frac{1}{2b}}.
\]
Concerning 
the minimization problem $J_\infty$, we can distinguish three cases.

\medskip

\underline{When $1 < q < \frac{d+2}{2}$}, the problem $J_\infty(\lambda)$ is bounded from below, and optimizers always exist, see for instance~\cite{Lio-84, Lio-85}. Up to a global translation, the optimizer for $J_\infty(\lambda)$ is unique, radial decreasing and positive. It satisfies the NLS equation for some $\mu = \mu(\lambda)$. By uniqueness, the map $\mu(\lambda)$ is the inverse map of $\lambda(\mu)$ defined previously. Note that in this case, the map $\mu \mapsto \lambda(\mu)$ is increasing (since $b> 0$).

\medskip

\underline{When $\frac{d+2}{d} < q < \frac{d}{d-2}$} ($q < \infty$ if $d = 1,2$), the problem $J_\infty(\lambda)$ is not bounded from below (so minimizers do not exist). However, the NLS equation still have a unique positive solution up to translation, for all $\mu > 0$. This solution is a critical point for the functional  $J_\infty$, but cannot be its optimizer since $J_\infty$ is unbounded from below. This time, the map $\mu \mapsto \lambda(\mu)$ is decreasing (since $b < 0$). 

\medskip

\underline{When $\frac{d}{d-2} < q $}, which can happen only in dimensions $d \ge 3$, the energy $J_\infty(\lambda)$ is not bounded from below, and the NLS equation does not have a solution. \\

\medskip

Our choice of $F_\alpha$ with $\frac{d+2}{d} < q$ somehow corresponds to the last two cases. The regularization done for $F_\alpha$ enforces the functional $J_\alpha(\lambda)$ to be bounded from below. However, minimizers might still not exist! In the case of the double power function $F(t) = - t^{2q} + t^{2p}$, which shares many properties with our case, it is proved that minimizers do exist iff $\lambda$ is sufficiently large, see~\cite{LewRot-20}.

\subsection{Idea of the proof in the fermionic case}

In the fermionic case, the problem with $F_\infty(t) = - t^q$ was studied in~\cite{GonLewNaz-21}. It is proved that for $q < \frac{d+2}{d}$, the problem with $N$ particules is well-posed, for an infinity of integers $N$, and in particular for $N = 2$. A crucial feature in the proof is that $F_\infty$ is a concave function, which allows to prove well-posedness by checking the strict binding inequality only over integers (see below). This is our main motivation for taking such a particular choice for $F_\alpha$: it is a concave function, which is sub-critical at infinity, so that the bosonic and fermionic energies are always bounded from below. 

\medskip

Our proofs of Theorems~\ref{th:main_bosonic_intro} and~\ref{th:main_fermionic_intro} use ideas that can be found in~\cite{LewRot-20} and~\cite{GonLewNaz-21}.

\subsection*{Acknowledgements} This work has received funding from CNRS-Africa Visiting Fellowships Program. S. Lahbabi warmly thanks CEREMADE for hosting her during the preparation of this work.


\section{General properties of the GS energies}

Let us first state some general properties, which hold for both the bosonic and fermionic problems. Recall that the functional $\cE_\alpha$ has been defined in~\eqref{eq:def:Ealpha}, as
\[
    \cE_\alpha(\gamma) :=  \Tr(- \Delta \gamma) + \int_{\R^d} F_\alpha(\rho_\gamma),
\]
and that we study the minimization problems
\begin{align*}
    J_\alpha(\lambda) & := \inf \left\{ \cE_\alpha(\gamma),  \quad 0 \le \gamma=\gamma^*, \quad \Tr(\gamma) = \lambda \right\}, \\
    I_\alpha(\lambda) & := \inf \left\{ \cE_\alpha(\gamma), \quad  0 \le \gamma=\gamma^* \le 1, \quad \Tr(\gamma) = \lambda \right\}.
\end{align*}

\begin{lemma}[General properties]
    \label{lem:general_facts}~
    \begin{enumerate}
        \item For all $\alpha > 0$ and all $\lambda > 0$, we have $J_\alpha(\lambda) \le I_\alpha(\lambda) \le 0$. In addition, for $0 \le \lambda \le 1$, we have the equality $J_\alpha(\lambda) = I_\alpha(\lambda)$. 
        \item {\bf (Weak binding inequality).} For all $\lambda, \lambda' \ge 0$ and all $\alpha > 0$, we have
        \[
            J_\alpha(\lambda + \lambda') \le J_\alpha(\lambda) + J_\alpha(\lambda'), \qquad \text{and} \qquad
            I_\alpha(\lambda + \lambda') \le I_\alpha(\lambda) + I_\alpha(\lambda').
        \]
        \item {\bf (Relaxation).} For all $\alpha > 0$, the maps $\lambda \mapsto I_\alpha(\lambda)$ and $\lambda \mapsto J_\alpha(\lambda)$ are non--increasing. In particular, the constraint $\Tr(\gamma) = \lambda$ can be relaxed, in the sense that we also have
        \begin{align*}
            J_\alpha(\lambda) & = \inf \left\{ \cE_\alpha(\gamma),  \quad 0 \le \gamma=\gamma^*, \quad \Tr(\gamma) \le \lambda \right\}, \\
            I_\alpha(\lambda) & = \inf \left\{ \cE_\alpha(\gamma), \quad  0 \le \gamma=\gamma^* \le 1, \quad \Tr(\gamma) \le \lambda \right\}.
        \end{align*}
    \end{enumerate}
\end{lemma}

\begin{proof}
    For the first point, we note that the only difference between the two problems is the Pauli condition $0 \le \gamma \le 1$ for the fermionic problem. In particular, the minimization set for the fermionic problem is a subset of the one for the bosonic problem, so
    \[
    \forall \alpha > 0, \qquad \forall \lambda > 0, \qquad J_\alpha(\lambda) \le I_\alpha(\lambda).
    \]
    If $0 \le \lambda \le 1$, then the two problems are equal since $\gamma \ge 0$ and $\Tr(\gamma) \le \lambda \le 1$ implies that $\gamma \le 1$ satisfies the Pauli principle. To see that $I_\alpha(\lambda) \le 0$, consider $0 \le \gamma = \gamma^* \le 1$ with $\Tr(\gamma) = \lambda$ a finite rank operator with $C^\infty_0$ eigenfunctions. For $s > 0$ a scaling parameter, consider
    \[
        \gamma_s (\bx, \by) = s^d \gamma( s \bx, s \by), \qquad \text{so that} \qquad \rho_s(\bx) := \gamma_s(\bx, \bx) = s^d \rho_\gamma(s \bx). 
    \]
    Then, for $s$ small enough, we have $\| \rho_s \|_\infty < \alpha$, so that $F_\alpha(\rho_s) = - \rho_s^q$ pointwise. We obtain, for such an $s$,
    \[
        \cE_\alpha(\gamma_s) = s^2 \Tr( - \Delta \gamma) - s^{d(q-1)} \int_{\R^d} \rho_\gamma^q,
    \]
    which goes to $0$ as $s \to 0$. So $I_\alpha(\lambda) \le 0$ for all $\lambda > 0$.
    
    \medskip
    
    Let us prove the weak--binding inequality. We only prove the result for the fermionic problem (the proof in the bosonic case is similar). We consider two minimizing sequences $(\gamma_n^\lambda)$ and $(\gamma_n^{\lambda'})$ for the problems $I_\alpha(\lambda)$ and $I_\alpha(\lambda')$. Without loss of generality, we may assume that these operators are finite rank with compactly supported eigenfunctions. Let $R_n > 0$ be large enough so that the supports of $\gamma_n^\lambda$ and $\tau_{R_n} \gamma_n^{\lambda'} \tau_{R_n}^*$ are disjoints, where  $\tau_{R_n} f (\bx):= f(\bx - R_n \be_1)$ is the translation operator in the $\be_1$ direction. Then, the operator
    \[
    \widetilde{\gamma}_n = \gamma_n^\lambda + \tau_{R_n} \gamma_n^{\lambda'} \tau_{R_n}^*,
    \]
    satisfies $\widetilde{\gamma}_n \ge 0$ and $\Tr(\widetilde{\gamma}_n) = \lambda + \lambda'$. In addition, if $\gamma_n^\lambda$ and $\gamma_n^{\lambda'}$ satisfy  Pauli principle, then so does $\widetilde{\gamma}_n$ thanks to our assumption on the support. So $\widetilde{\gamma}_n$ is a valid test operator for the problem $I_\alpha(\lambda + \lambda')$, and we get
    \[
    I_\alpha(\lambda + \lambda') \le \cE_\alpha(\widetilde{\gamma}_n)  =  \cE_\alpha(\gamma_n^\lambda) + \cE_\alpha(\gamma_n^{\lambda'}).
    \]
    Letting $n$ go to infinity proves the weak--binding inequality.

    \medskip
    
    The weak--binding inequality together with the fact that $I_\alpha(\lambda) \le 0$ and $J_\alpha(\lambda) \le 0$ implies that the maps $\lambda \mapsto I_\alpha(\lambda)$ and $\lambda \mapsto I_\alpha(\lambda)$ are non--increasing. 
\end{proof}

We show below that the bosonic map $\lambda \mapsto J_\alpha(\lambda)$ is continuous by scaling  (see Lemmas~\ref{lem:scaling_bosonic}--\ref{lem:props-J} below). It is not straightforward that the fermionic map $\lambda \mapsto I_\alpha(\lambda)$ is continuous, and we postpone this question to Section~\ref{sec:fermionic}. However, we can prove continuity in the parameter $\alpha$ in both cases, using a rescaling which preserves the Pauli principle.

\begin{lemma}[Rescaling]
    \label{lem:rescaling_beta}
    Introduce the modified energy 
    \[
       \boxed{  \widetilde{\cE}_\beta(\gamma) := \Tr(- \Delta \gamma) + \beta \int_{\R^d} F_1(\rho_\gamma), }
    \]
    together with the corresponding bosonic/fermionic problems $\widetilde{J}_\beta(\lambda)$ and $\widetilde{I}_\beta(\lambda)$. Then, for all $\lambda \ge 0$, we have
    \[
        J_\alpha(\lambda) = \alpha^{2/d} \widetilde{J}_\beta(\lambda), \quad \text{and} \quad I_\alpha(\lambda) = \alpha^{2/d} \widetilde{I}_\beta(\lambda), \quad \text{with} \quad
        \beta = \alpha^{q - \frac{d+2}{d}}.
    \]
    As a consequence, for all $\lambda > 0$, the maps $\alpha \mapsto J_\alpha(\lambda)$ and $\alpha \mapsto I_\alpha(\lambda)$ are continuous and non--increasing.
\end{lemma}

\begin{proof}
      For $0 \le \gamma ( \le 1)$, introduce the rescaled operator $\widetilde{\gamma}(\bx, \by) := \alpha \gamma (\alpha^{1/d} \bx, \alpha^{1/d} \by)$, so that $\widetilde{\rho}(\bx) = \alpha \rho( \alpha^{1/d} \bx)$. Then we have $0 \le \widetilde{\gamma} (\le 1)$,
      $ \Tr ( - \Delta \widetilde{\gamma}) = \alpha^{2/d}\Tr ( - \Delta \gamma)$, and, with our scaling~\eqref{eq:scaling_Falpha} for $\alpha \mapsto F_\alpha$, we get
      \[
      \int_{\R^d} F_\alpha(\widetilde{\rho}) =  \int_{\R^d} \alpha^q F_1( \rho(\alpha^{1/d} \bx)) \rd \bx = \alpha^{q - 1} \int_{\R^d} F_1( \rho).
      \]
      So, with $\beta := \alpha^{q - \frac{d+2}{d}}$, we get $\cE_\alpha(\gamma)  = \alpha^{2/d} \widetilde{\cE}_\beta(\widetilde{\gamma})$. This proves that $J_\alpha = \alpha^{2/d} \widetilde{J}_\beta$ and $I_\alpha = \alpha^{2/d} \widetilde{I}_\beta$.
      
      \medskip
      
       The maps $\beta \mapsto \widetilde{\cE}_\beta(\gamma)$ are linear non--increasing (since $F_1 \le 0$ pointwise), so the maps $\beta \mapsto \widetilde{J}_\beta$ and $\beta \mapsto \widetilde{I}_\beta$ are concave non--increasing, hence continuous. This proves that the original maps $\alpha \mapsto J_\alpha(\lambda)$ and $\alpha \mapsto I_\alpha(\lambda)$ are continuous as well. In addition, the map $\alpha\mapsto F_\alpha(t)$ is non--increasing, 
       thus $\alpha \mapsto J_\alpha(\lambda)$ and $\alpha \mapsto I_\alpha(\lambda)$ are also  non--increasing. 
\end{proof}

\begin{remark}
    While our choice for $F_\alpha$ in~\eqref{eq:scaling_Falpha} might seem arbitrary at first, it is equivalent (up to scaling) to the choice $\widetilde{F}_\beta := \beta F_1$. In this case, $\beta$ somehow measures the strength of the non--linear self-interaction. We find the setting with $F_\alpha$ more transparent, as we recover the (not well-posed) NLS problem in the limit $\alpha \to \infty$. Recall that $\beta = \alpha^{q - \frac{d+2}{d}}$ with $q > \frac{d+2}{d}$. So $\beta \to 0$ iff $\alpha \to 0$, and $\beta \to \infty$ iff $\alpha \to \infty$.
\end{remark}

From the fact that the maps $\alpha \mapsto J_\alpha(\lambda)$ and $\alpha \mapsto I_\alpha(\lambda)$ are non--increasing and continuous, while $\lambda \mapsto J_\alpha(\lambda)$ and $\lambda \mapsto J_\alpha(\lambda)$ are non--increasing, we obtain the following result.

\begin{lemma}[Critical $\alpha$]
    \label{lem:critical_alpha}
   For all $\lambda > 0$, there are $\alpha_c^J(\lambda) \in [0, \infty]$ and $\alpha_c^I(\lambda) \in [0, \infty]$ so that
    \[
         J_\alpha(\lambda) = 0 \quad \text{iff} \quad 0 \le \alpha \le \alpha_c^J(\lambda), \quad \text{and} \quad 
         I_\alpha(\lambda)= 0 \quad \text{iff} \quad 0 \le \alpha \le \alpha_c^I(\lambda).
    \]
    In addition, the maps $\lambda \mapsto \alpha_c^{J}(\lambda)$ and $\lambda \mapsto \alpha_c^I(\lambda) \in [0, \infty]$ are non--increasing. Finally, for all $0 \le \lambda \le 1$, we have $\alpha_c^J(\lambda)= \alpha_c^I(\lambda)$.
\end{lemma}

It is unclear yet that $\alpha_c^{I/J}(\lambda) \neq 0$ and $\alpha_c^{I/J}(\lambda) \neq \infty$, but we will prove below that this is the case, so $0 < \alpha_c^{I/J}(\lambda) < \infty$. We will see below that for $\alpha < \alpha_c^{J}(\lambda)$, we have $J_\alpha(\lambda) = 0$ and the problem is not well-posed: minimizers do not exist. The problem becomes well-posed only for $\alpha$ large enough.

With these general remarks at hand, we can now focus on the specificities of the bosonic and fermionic problems. We start with the study of the bosonic problem.

\section{The bosonic problem}

\subsection{Properties of $J_\alpha(\lambda)$}
\label{sec:bosons_Jalpha}

It is useful to change variable for this case. In order to do so, we recall the Hoffman-Ostenhof inequality~\cite{HofHof-77}, which states that for all $\gamma \in \cS(L^2(\R^d))$ with $\gamma \ge 0$, we have
\[
    \Tr\bra{-\Delta \gamma} \ge \int_{\R^d} | \nabla \sqrt{\rho_\gamma} |^2.
\]
In addition, given $\rho \in L^1(\R^d)$, we have equality for the rank-one operator $\gamma = | u \rangle \langle u |$, with $u := \sqrt{\rho}$. So, changing to the $u$ variable, our minimization problem becomes
\begin{equation} \label{eq:def:Falpha}
    J_\alpha(\lambda) = \inf \left\{ \cF_\alpha(u) , \ u \in L^2(\R^d) , \ \int_{\R^d} | u |^2 = \lambda \right\}, \quad 
    \text{with} \quad
    \cF_\alpha(u) := \int_{\R^d} | \nabla u |^2 + \int_{\R^d} F_\alpha (u^2).
\end{equation}

For the bosonic problem, we can use an extra scaling, namely $u = \lambda^{1/2} v$ with $\| v \|_2 = 1$. This scaling has no equivalent in the fermionic setting, as it does not preserve the Pauli condition. This allows to reduce the study of $J_\alpha(\lambda)$ to the one of $J_{\alpha = 1}(\lambda)$. More specifically, we have the following result.
\begin{lemma}[Scaling properties for the bosonic problem]
    \label{lem:scaling_bosonic}
    We have
    \[
        J_\alpha(\lambda) = \alpha^{ q (1 - \frac{d}{2}) + \frac{d}{2}} J_1 \left( \alpha^{\frac{d}{2}(q - \frac{d+2}{d} )} \lambda \right).
    \]
    In addition, $u$ is a minimizer of $J_\alpha(\lambda)$ iff it is of the form
    \[
        u(\bx) =  \sqrt{L} s^{\frac{d}{2}} v(s \bx) \qquad \text{with} \quad L = \alpha^{-\frac{d}{2}(q - \frac{d+2}{d} )}, \quad s = \alpha^{\frac{q-1}{2}},
    \]
    and with $v$ a minimizer of $J_{\alpha = 1}(\lambda/L)$.
\end{lemma}

\begin{proof}
    A straightforward computation shows that, with $u(\bx) =  \sqrt{L} s^{\frac{d}{2}} v(s \bx)$, we have
    \begin{equation} \label{eq:scaling_bosonic}
        \cF_\alpha(u) = L s^2 \int_{\R^d} | \nabla v |^2 + \frac{\alpha^q}{s^d} \int_{\R^d} F_1 \left( \frac{L s^d}{\alpha} v^2 \right), 
        \quad \text{and} \quad \int_{\R^d} | u |^2 = L \int_{\R^d} | v |^2.
    \end{equation}
    The quantities $L$ and $s$ are chosen so that $L s^{2} = \alpha^q s^{-d}$ and $L s^d = \alpha$. 
\end{proof}

It is therefore enough to study the problem with $\alpha = 1$ to deduce the properties of $J_\alpha$ for all $\alpha$. In what follows, we restrict our attention to $\alpha = 1$. 
\begin{lemma}\label{lem:props-J}
    The map $\lambda \mapsto J_1 (\lambda)$ is continuous, non--positive and non--increasing. In addition, if $d \ge 2$, this map is concave.
\end{lemma}

\begin{proof}
    We have already proved in Lemma~\ref{lem:rescaling_beta} that $\alpha \mapsto J_\alpha(\lambda)$ is continuous, from which we already infer that $\lambda \mapsto J_1(\lambda)$ is also continuous, since, by Lemma~\ref{lem:scaling_bosonic}, we have
    \[
        J_1(\lambda) = J_\alpha(1) \alpha^{- q(1 - \frac{d}{2}) - \frac{d}{2}} \quad \text{with} \quad \alpha = \alpha(\lambda) =  \lambda^{\frac{2}{d} \frac{1}{q - \frac{d+2}{d}}}.
    \]
    The fact that it is non--positive and non--increasing was proved in Lemma~\ref{lem:general_facts}. It remains to prove concavity. We use yet another change of variable, which is a combination of the previous ones. We write $u(\bx) =  \sqrt{L} s^{\frac{d}{2}} v(s \bx)$, but this time with the choices $L = \lambda$ and $s = \lambda^{-1/d}$. In particular, from~\eqref{eq:scaling_bosonic}, we get $\| v \|_2^2 = 1$, and
    \[
        \cF_{\alpha = 1}(u) = \lambda^{\frac{d - 2}{d}} \left( \int_{\R^d} | \nabla v |^2 + \beta \int_{\R^d} F_1(v^2) \right),  \quad \text{with} \quad \beta := \lambda^{2/d}.
    \]
    Minimizing over $u \in L^2(\R^d)$ with $\int_{\R^d} | u |^2 = \lambda$ gives $J_1(\lambda) = \lambda^{\frac{d-2}{d}} \widetilde{J}_\beta(1)$, where $\widetilde{J}_\beta(1)$ is the infimum of the functional in parenthesis over $v \in L^2(\R^d)$ with $\int_{\R^d} | v |^2 = 1$. As in the proof of Lemma~\ref{lem:rescaling_beta}, the map $\beta \mapsto \widetilde{J}_\beta(1)$ is negative, non--increasing and concave in $\beta$. This implies that $J_1(\lambda)$ is of the form
    \[
        J_1(\lambda) = \lambda^{1 - \frac{2}{d}} f(\lambda^{\frac{2}{d}}),
    \]
    with $f$ a non--positive non--increasing concave function. In the case where $d\geq 2$, this implies that $J_1(\lambda)$ is also a non--positive non--increasing concave function. For instance, in the case where $f$ is smooth, we can differentiate twice, and get
    \[
        J_1''(\lambda) = \frac{4}{d^2} \lambda^{\frac{2}{d} - 1} f''(\lambda^{\frac{2}{d}}) + \frac{2(d-2)}{d^2} \frac{1}{\lambda^{1 + \frac{d}{2}}} \left[ \lambda^{\frac{2}{d}} f'(\lambda^{\frac{2}{d}}) - f(\lambda^{\frac{2}{d}})  \right].
    \]
    The first term is non--positive, since $f$ is concave. In addition, since $f(0) = 0$, we have $\lambda^{2/d} f'(\lambda^{2/d}) \le f(\lambda^{2/d})$ by concavity again. So the second term is also non--positive, hence $J_1'' \le 0$ and $J_1$ is concave. The result follows.
\end{proof}

\begin{remark}
	The results of this section do not depend on the behaviour of $F_1$ in $+\ii$. 
\end{remark}

\subsection{Existence of a minimizer for $\lambda >\lambda_c$}

We now focus on the existence of a minimizer for the problem $J_{\alpha = 1}(\lambda)$ problem. We start with the following lower bound.

\begin{lemma}[The energy is bounded from below]
    There is a constant $C \ge 0$ so that, for all $\lambda > 0$, we have
    \[
        J_1(\lambda) \ge - C \lambda^{\frac{d - r(d-2)}{d + 2 - dr}}.
    \]
    In particular, the energy is bounded from below.
\end{lemma}

\begin{proof}
    The assumption on the function $F_1$ shows that there is a constant $C$ so that $F_1(t) \ge - C t^r$. In particular
    \[
        \cF_1(u) \ge   \int_{\R^d} | \nabla u |^2 - C \int_{\R^d} |u |^{2r}.
    \]
    The Gagliardo--Niremberg inequality then states that for all $1 \le r < \frac{d}{d - 2}$ ($1 \le r < \infty$ in dimensions $d = 1,2$), there is a constant $K_{r,d} > 0$ so that, for all $u \in H^1(\R^d)$, we have
    \[
        \| u \|_{2r} \le K_{r, d} \|\nabla u \|_2^{\theta} \| u \|_2^{1 - \theta}, \quad \text{with} \quad \theta := \frac{d}{2}\frac{r-1}{r}.
    \]
    Setting $X = \| \nabla u \|_2^2$, we get
    \begin{equation} \label{eq:F1_bounded_from_below}
        \cF_1(u) \ge \| \nabla u \|_2^2 - C K_{r, d}^{2r} \lambda^{r(1 - \theta)} \| \nabla u \|_2^{2 r \theta}
        \ge \inf_{X > 0} \left\{ X - C K_{r, d}^{2r} \lambda^{r(1 - \theta)}X^{r \theta} \right\}.
    \end{equation}
    The right--hand side is bounded from below whenever $r \theta < 1$, which is our condition $r < \frac{d+2}{d}$. Under this condition, this minimum is obtained for $X$ satisfying $C K_{r,d}^{2r} \lambda^{r(1 - \theta)} r \theta X^{r \theta - 1} =\frac{1}{r\theta}$, and the result follows.
\end{proof}

It is unclear at this point that there is $\lambda > 0$ so that $J_1(\lambda) < 0$ (it could be that $J_1(\lambda) = 0$ for all $\lambda \ge 0$).  However, this is not possible, thanks to the following result.

\begin{lemma}
    For $\lambda$ large enough, we have $J_1(\lambda) < 0$.
\end{lemma}

\begin{proof}
    For this proof, it is easier to prove that $J_\alpha(\lambda) < 0$ for some (large) value of $\alpha > 0$ and for some $\lambda > 0$; the result then follows from the scaling in Lemma~\ref{lem:scaling_bosonic}. Indeed, consider a function $u \in C^\infty_0(\R^d)$ such that
    \[
        \int_{\R^d} | u |^{2q} >  \int_{\R^d} | \nabla u |^2.
    \]
    Such a function exists, since $q > 1$ (consider $u_s = su$ with $s$ large). Now take $\alpha > \max | u |^2$, so that $F_\alpha(| u |^2) = - | u |^{2q}$. For this value of $\alpha$, we have $\cE_\alpha(u) < 0$, which proves that $J_\alpha(\lambda) < 0$ for $\lambda :=  \| u \|_2^2$.
\end{proof}

We can therefore define the bosonic critical mass $\lambda_c \in [0, \infty)$ by
\begin{equation} \label{eq:def:lambdac}
    \lambda_c := \inf \left\{ \lambda > 0, \quad J_1(\lambda) < 0\right\}.
\end{equation}
We will prove below that $\lambda_c > 0$. It is the critical mass after which the function $\lambda \mapsto J_1(\lambda)$ becomes strictly negative, while $J_1(\lambda) = 0$ for $\lambda \in [0, \lambda_c]$. According to Lemma~\ref{lem:scaling_bosonic}, it is linked to the map $\alpha_c^J(\lambda)$ defined in~\eqref{lem:critical_alpha} via the relation
\[
    \alpha_c^J(\lambda) = \left( \frac{\lambda_c}{\lambda} \right)^{\frac{2}{d} \frac{1}{q - \frac{d+2}{d}}}.
\]

We now turn to the existence of minimizers. In what follows, we restrict our attention to the case $d \ge 2$. The following result is rather classical.
\begin{lemma}
    \label{lem:existence_minimizer_bosonic_lambda>lambda_c}
    Assume $d \ge 2$. If $\lambda > \lambda_c$, then $J_1(\lambda)$ has a minimizer and $J_1(\lambda)<0$.
\end{lemma}
Note the strict inequalities. Below, we will prove that $J_1(\lambda_c)$ also has a minimizer,  although $J_1(\lambda_c)= 0$.

\begin{proof}
    Consider $(u_n)$ a minimizing sequence for $J_1(\lambda)$ (so $(u_n)$ is bounded in $L^2(\R^d)$ with $\| u_n \|_2 = \lambda$).  By re-arrangement, we may assume without loss of generality that all functions $u_n$ are positive radial non--increasing. The inequality~\eqref{eq:F1_bounded_from_below} shows that the sequence $(u_n)$ is bounded in $H^1_{\rm rad}(\R^d)$, the set of radial functions in $H^1(\R^d)$. By Banach--Alaoglu theorem, we may extract a sub-sequence, still noted $(u_n)$, which converges weakly to some $u_*$ in $H^1_{\rm rad}(\R^d)$.
    
    \medskip
    
    Recall Strauss' theorem~\cite{Str-77} (see also~\cite[Lemma A.II]{BerLio-83}), which states that, for $d \ge 2$, the radial embedding $H^1_{\rm rad}(\R^d) \to L^p(\R^d)$ is compact for $2 < p < \frac{2d}{d - 2}$. In particular, $(u_n)$ converges strongly to $u_*$ in $L^{2r}(\R^d)$, and, up to another sub-sequence, we may also assume that $u_n \to u_*$ pointwise, and that there is a domination $G \in L^{2r}(\R^d)$ so that $| u_n | \le G$ for all $n$.
    
    \medskip
    
    Using the weak lower semi-continuity of the norm, we have
    \[
        \int_{\R^d} | \nabla u_* |^2 \le \liminf_n \int_{\R^d} | \nabla u_n |^2.
    \]
    On the other hand, since $F_1$ is continuous, the functions $F_1( u_n^2)$ converges pointwise to $F_1(u^2)$. In addition,  $| F_1(t) | \le C t^{r}$ for some constant $C$, thus we have the domination
    \[
        \left| F_1(u_n^2) \right| \le C | u_n |^{2r} \le C G^{2r},
    \]
    which is integrable. The dominated convergence theorem then shows that $\displaystyle  \int_{\R^d} F(u_*^2) = \lim_{n \to \infty}  \int_{\R^d} F(u_n^2)$, so
    \[
        \cF_1(u_*) \le \liminf_n \cF_1(u_n) = J_1(\lambda).
    \]
    Letting $\lambda' := \| u_* \|_2^2$, we have $0 \le \lambda' \le \lambda$ by the weak lower semi--continuity of the $L^2$-norm. This already proves that $u_*$ is a minimizer for the relaxed problem (with condition $\| u \|_2^2 \le \lambda$). In addition, since $J_1(\lambda)$ is concave and strictly negative for $\lambda > \lambda_c$, it is strictly decaying for $\lambda > \lambda_c$. So, in this case, we have
    \[
        J_1(\lambda') \le \cF_1(u_*) \le J_1(\lambda) \le J_1(\lambda'),
    \]
    which proves that $\lambda = \lambda'$. Hence, if $\lambda > \lambda_c$, $u_*$ is a minimizer for $J_1(\lambda)$, with $\| u_* \|_2^2 = \lambda$.
\end{proof}

\subsection{Euler-Lagrange equation}
In this section, we assume that $d \ge 2$ and $\lambda > \lambda_c$. Let $u$ be a radial non--increasing positive minimizer. It satisfies the Euler--Lagrange equation
\begin{equation} \label{eq:EL_mu}
- \Delta u  = g_\mu(u), \quad \text{with} \quad g_\mu(t) := - t \left( F_1'(t^2) + \mu \right), \quad \text{for some} \quad \mu \in \R.
\end{equation}
Actually, since $u$ is also the minimizer for the relaxed problem $\| u \|_2^2 \le \lambda$, the Lagrange multiplier satisfies $\mu \ge 0$.

\medskip

The goal of this section is to study this semi--linear equation. More specifically, for all $\mu > 0$, we will prove that there is a unique ground state solution to~\eqref{eq:EL_mu}. Given $\lambda > 0$ and a minimizer $u_\lambda$ for the problem $J_1(\lambda)$, we know that there is a $\mu \ge 0$ for which $u_\lambda$ is a solution to~\eqref{eq:EL_mu}, but there could {\em a priori} be another minimizer $v_\lambda$ corresponding to another $\mu$. So, unfortunately, uniqueness of solutions of~\eqref{eq:EL_mu} does not imply uniqueness for our original problem, see also~\cite{LewRot-20}.

\medskip

Recall that $u$ is called a {\em ground state} solution to~\eqref{eq:EL_mu} if it is positive and decaying to $0$ at  infinity. By moving plane techniques~\cite{GidNiNir-79}, such a solution is radial decreasing, up to a translation. In what follows, we focus on positive radial decreasing solutions of~\eqref{eq:EL_mu}.

\subsubsection{Existence of Ground states}

\begin{proposition}
    \label{prop:existence_ground_state}
    For all $\mu > 0$, there is a unique ground state solution to~\eqref{eq:EL_mu}. This solution is of class $C^3$, and satisfies the following pointwise bound: there is $0 < c_- \le c_+ < \infty$ so that for any $\bx \in \R^d$,
    \[
       c_- \leq (1 + | \bx |^{\frac{d-1}{2}}) \re^{\sqrt{\mu} | \bx |} u_\mu(\bx)\leq c_+  , 
       \quad \text{and} \quad
      (1 + | \bx |^{\frac{d-1}{2}}) \re^{\sqrt{\mu} | \bx |} | \nabla u_\mu   | (\bx) \le c_+ .
    \]
\end{proposition}

\begin{proof}
    \underline{{\bf Existence.}} This is a consequence of a result by Berestycki and Lions~\cite[Theorem 1]{BerLio-83} for $d\geq 3$. Our function $g_\mu$ is explicitly given by
\[
    g_\mu(t) = \begin{cases}
        - t \left( -q t^{2(q-1)} + \mu \right) & t \le 1 \\
        -t \left( \fb - \fc r t^{2(r-1)} + \mu \right) & t \ge 1.
    \end{cases}
\]
It satisfies $g_\mu(-t) = - g_\mu(t)$, and
\begin{equation}\label{eq:cd-BerLio}
  \lim_{t\to 0} \frac{g_\mu(t)}{t} = - \mu <0, \qquad \text{and} \qquad \lim_{t\to \ii} \frac{g_\mu(t)}{t^{\frac{d+2}{d-2}}} = \lim_{t \to \infty} \fc t^{-2(\frac{d}{d-2} - r)} = 0.
\end{equation}
Finally, since $g_\mu(t) \to +\infty$ as $t \to \infty$, the function
\[
    G_\mu(t) := \int_0^t g_\mu(s) \rd s =\left\{ \begin{array}{ll}
    	\frac{1}{2}t^{2q}-\frac{\mu}{2} t^2 & t\leq 1\\
    	\frac{1}{2}\bra{\fa+ \fc  t^{2r} -\bra{\fb+\mu}t^2}& t\geq 1
    \end{array}\right.
\]
becomes positive for $t$ large enough. So the hypotheses of~\cite[Theorem 1]{BerLio-83} are satisfied. We deduce that~\eqref{eq:EL_mu} has a continuous ground state solution $u_\mu$, which is positive $u_\mu > 0$, radial decreasing, of class $C^2$, and satisfying the exponential bound
\begin{equation} \label{eq:first_bound_umu}
    | D^n u |(\bx) \le C \re^{- \delta | \bx |}
\end{equation}
for some $C \ge 0$, $\delta > 0$ and for $n = (n_1, \cdots, n_d) \in \N_0^d$ with $n_1 + \cdots + n_d \le 2$. In addition, since our function $g_\mu$ is of class $C^1$, we get the further regularity that $u_\mu$ is of class $C^3$. In $d = 2$, we rather use~\cite{BerGalOta-83}. In this case, it is enough to replace the power  $\frac{d+2}{d-2}$ in~\eqref{eq:cd-BerLio} by any power $\ell$, and we take $\ell > 2r-1$. 

\medskip

\underline{{\bf Uniqueness.}}  To prove uniqueness, we use a result of McLeod~\cite[Theorem 1]{McL-93} (see also Serrin--Tang~\cite[Theorem 1.1]{SerTan-00} for dimension $d\ge 3$). We check that $g_\mu$ is $C^1$, satisfies $g_\mu(t) \le 0$ for $t \in (0, T]$, and $g_\mu(t) > 0$ for $t > T$, with
\[
    T = \begin{cases}
        \left( \frac{\mu}{q} \right)^{\frac{1}{2(q-1)}} & \quad \text{if}  \quad 0 < \mu < q ; \\
        \left( \frac{\fb + \mu}{\fc r} \right)^{\frac{1}{2(r-1)}} = \left( \frac{q(q-r) + \mu (r-1)}{q(q-1)} \right)^{\frac{1}{2(r-1)}} & \quad \text{if} \quad \mu \ge q.
    \end{cases}
\]
In addition, we consider the function defined, for any $\lambda>0$, by $I_{\lambda}(t):=\lambda t g_\mu'(t)-(\lambda+2)g_\mu(t)$ and the auxiliary function $k_\mu(t):=\frac{tg'_\mu(t)}{g_\mu(t)}$
. Explicitly, we have
\[
    k_\mu(t) = \begin{cases}
        (2q - 1) + \frac{2(q-1) \mu}{q t^{2(q-1)} - \mu} & \quad t \le 1 \\
         (2r - 1) + \frac{2(r-1)(\fb + \mu)}{\fc r t^{2(r-1)} - (\fb + \mu)} & \quad t \ge 1.
    \end{cases}
\]
which has a singularity at $t = T$, and $I_{\lambda}(t)=\lambda g_\mu(t)\left(k_\mu(t)-1-\frac{2}{\lambda}\right)$.

Using the explicit formula, we can easily remark that $k_\mu(0)=1$, $\lim_{t\to T^\pm}k_\mu(t)=\pm\infty$ and $\lim_{t\to +\infty}k_\mu(t)=2r-1>1$. Moreover, $k_\mu$ is decreasing on $(0,T)$ and on $(T,\infty)$. On the one hand, this implies that $I_{\lambda}<0$ on $(0,T)$ for all $\lambda>0$. On the other hand, we deduce that for any $S>T$, there exists a unique $\lambda(S)>0$ such that $k_\mu(S)=1+\frac{2}{\lambda(S)}$. As a consequence, $I_{\lambda(S)}>0$ on $(T,S)$ and $I_{\lambda(S)}<0$ on $(S, \infty)$.

So the hypothesis of~\cite[Theorem 1]{McL-93} are satisfied, and we deduce that ground state is unique.

\medskip

\underline{{\bf Decay at infinity.}} The decay estimates for the function $u_\mu$ can be deduced from~\cite{HofHofSwe-85}. Indeed, the function $u_\mu$ satisfies the Schrödinger equation
\[
    (- \Delta + V + \mu) u = 0, \quad \text{with} \quad V := F_1'(u_\mu^2).
\]
Since $u_\mu$ is continuous and radial, so is the potential $V$. In addition the crude estimate~\eqref{eq:first_bound_umu} shows that this potential $V$ decays exponentially fast as $\bx \to \infty$. We are therefore in the setting of~\cite[Corollary 2.2 (ii)]{HofHofSwe-85} with $R = 0$, and we deduce the existence of $0 < c_- \le c_+ < \infty$ so that 
\[
    \forall \bx \in \R^d, \qquad  c_- \le  | \bx |^{\frac{d-1}{2}}\re^{\sqrt{\mu} | \bx |} u_\mu(\bx) \le c_+.
\]
In addition, since $u_\mu$ is continuous (hence bounded) at the origin, we can replace $| \bx |^{\frac{d-1}{2}}$ by $(1 + | \bx |^{\frac{d-1}{2}})$, up to modifying the constants $c_\pm$. 

\medskip

It remains to prove the bound for $| \nabla u |$. The proof is similar to~\cite[Lemma 19]{GonLewNaz-21}. Note that since $\mu > 0$, the operator $(- \Delta + \mu)$ is invertible, and we can write
\[
    u_\mu = - ( - \Delta + \mu)^{-1} (V u_\mu), \quad \text{so} \quad u_\mu = - Y_{\sqrt{\mu}} * (Vu_\mu), \quad \text{and} \quad
    \nabla u_\mu = - (\nabla Y_{\sqrt{\mu}}) * (V u_\mu),
\]
where $Y_{\sqrt{\mu}}$ is the usual $d$--dimensional Yukawa potential, satisfying $(- \Delta + \mu) Y_{\sqrt{\mu}} = \delta_0$ in the distributional sense. In Fourier, this is
\[
    \widehat{Y}_{\sqrt{\mu}}(\bk) := \frac{1}{(2 \pi)^{d/2}} \int_{\R^d} Y_{\sqrt{\mu}}(\bx) \re^{- \ri \bk \cdot \bx} \rd \bx = \frac{1}{(2 \pi)^{d/2}} \dfrac{1}{\mu + | \bk |^2}.
\]
Since $Y_{\sqrt{\mu}}$ solves $(- \Delta + \mu) Y_{\sqrt{\mu}} = 0$ on $\cB(0, R)^c$ for any  $R > 0$, the previous bounds apply to the Yukawa potential, and there is $0 < c_1 \le c_2 < \infty$ so that, for all $\bx \in \R^d$, we have~\cite[Formula 9.7.2]{AS64}
\[
    c_1  \frac{\re^{- \sqrt{\mu} | \bx |}}{| \bx |^{\frac{d-1}{2}}} \le Y_{\sqrt{\mu}}(\bx) \le c_2 \frac{\re^{- \sqrt{\mu} | \bx |}}{| \bx |^{\frac{d-1}{2}}}.
\]
In addition, we also have $(- \Delta + \mu) (\partial_i Y_{\sqrt{\mu}}) = 0$ on $\cB(0, R)^c$ for any  $R > 0$ and any $1 \le i \le d$, so we also have
\[
    | \nabla Y_{\sqrt{\mu}} |(\bx) \le c_3 \frac{\re^{- \sqrt{\mu} | \bx |}}{| \bx |^{\frac{d-1}{2}}} \le c_4 Y_{\sqrt{\mu}}(\bx).
\]
For some $c_3, c_4 \ge 0$. On the other hand, since $| F_1' (t^2) | \le q t^{2(q-1)}$, we have the pointwise bound
\[
    | V u_\mu | (\bx) \le  q | u_\mu |^{2q - 1} (\bx) \le C  \frac{\re^{-\sqrt{\mu}(2q - 1) | \bx |}}{ (1 + | \bx |^{\frac{d-1}{2}})^{2q-1}}  \le \widetilde{C} Y_{(2q-1) \sqrt{\mu}}(\bx) .
\]
This gives the pointwise bound
\[
    | \nabla u_\mu | \le | \nabla Y_{\sqrt{\mu}} |* | V u_\mu | \le  c_4 \widetilde{C} \left(  Y_{\sqrt{\mu}}* Y_{(2q-1) \sqrt{\mu}} \right).
\]
The last convolution can be computed in Fourier space. From the equality
\[
    \frac{1}{(| \bk |^2 + a^2)} \frac{1}{(| \bk |^2 + b^2)} = \frac{1}{b^2 - a^2}\left( \frac{1}{|\bk|^2 + a^2} - \frac{1}{| \bk |^2 + b^2} \right),
\]
we get $Y_a * Y_b = \frac{1}{(2 \pi)^{d/2}}\frac{1}{b^2 - a^2} \left( Y_a - Y_b  \right)$, so 
\[
    Y_{\sqrt{\mu}}* Y_{(2q-1) \sqrt{\mu}} = \frac{1}{\mu (2 \pi)^{d/2}} \frac{1}{(2q - 1)^2 - 1} \left( Y_{\sqrt{\mu}} - Y_{(2q-1)\sqrt{\mu}} \right).
\]
Since $2q - 1 > 1$, the function $Y_{(2q-1)\sqrt{\mu}}$ decays faster than $| Y_{\sqrt{\mu}}|$, and the result follows.
\end{proof}

In what follows, we denote by $u_\mu$ the unique ground state of~\eqref{eq:EL_mu}. We also define 
\[
    \Lambda(\mu) := \int_{\R^d} | u_\mu |^2.
\]
Clearly, if $u$ is a minimizer of $J_1(\lambda)$, then the corresponding Lagrange multiplier $\mu$ is such that $\Lambda(\mu) = \lambda$. However, it is unclear if the map $\mu \mapsto \Lambda(\mu)$ is one-to-one, and we may have several candidates $u_{\mu_1}, u_{\mu_2}, ...$ with $\lambda = \Lambda(\mu_1) = \Lambda(\mu_2) = ...$ to be the minimizer(s) of $J_1(\lambda)$. In particular, it is unclear that the minimizer of $J_1(\lambda)$ is always unique.


\subsubsection{Behaviour of the ground state solution as $\mu \to 0$ and $\mu \to \infty$}

In Appendix~\ref{sec:appendix:asymptotics} below, we perform a detailed analysis of the function $u_\mu$ in the limits $\mu \to 0$ and $\mu \to \infty$. For the sake of clarity, we postpone this lengthy analysis, and just state a brief summary of the main points that we will use in the sequel.

\begin{proposition}[Asymptotics for $u_\mu$]
    \label{prop:mu1_mu2}
    The map $\mu \mapsto \Lambda(\mu)$ is continuous on $(0, \infty)$.
    \begin{enumerate}
        \item {\bf (Limit $\mu \to 0$)}. There is $\mu_{\rm min} > 0$ so that, for all $0 < \mu < \mu_{\rm min}$, we have $\cF_1(u_\mu) > 0$. In particular, for $0 < \mu < \mu_{\rm min}$, $u_\mu$ is not an optimizer for the problem $J_1(\lambda)$, for any $\lambda > 0$.
        \item {\bf (Limit $\mu \to \infty$)}. We have $\lim_{\mu \to \infty} \Lambda(\mu) = + \infty$.
    \end{enumerate}
\end{proposition}

We note that the continuity of $\mu \mapsto \Lambda(\mu)$ is a standard fact which can be proved as in~\cite{KilOhPoc-17}.

\subsection{Existence of a minimizer for $\lambda = \lambda_c$}

We now go back to our minimization problem $J_1(\lambda)$. Recall that in Lemma~\ref{lem:existence_minimizer_bosonic_lambda>lambda_c}, we have proved that $J_1(\lambda)$ has a minimizer for all $\lambda > \lambda_c$, where $\lambda_c$ was defined in~\eqref{eq:def:lambdac}.

\begin{lemma} \label{lem:lambda_c}
    The critical value $\lambda_c$ is strictly positive. Moreover, \\ 
    For all $0 < \lambda < \lambda_c$, we have $J_1(\lambda) = 0$, and the problem has no minimizers. \\
    For all $\lambda = \lambda_c$, we have $J_1(\lambda_c) = 0$, and the problem has a minimizer.
\end{lemma}

\begin{remark} \label{rem:rescaled_existence_bosonic}
    Going back to the original problem $J_\alpha(\lambda)$ by rescaling, we get that for all $\alpha \ge \alpha_c^J(\lambda)$, the problem $J_\alpha(\lambda)$ has a minimizer, while for $\alpha < \alpha_c^J(\lambda)$, the problem has no minimizers. 
\end{remark}

\begin{proof}  
     Let us first prove that $\lambda_c > 0$. The function $f: t \mapsto \frac{| F_1(t)|}{t^{\frac{d+2}{d}}}$ is continuous, positive, and satisfies $\lim_{t \to 0} f(t) = \lim_{t\to \infty} f(t) = 0$, so $f$ is bounded from above by a constant $M > 0$. On the other hand, we have the Gagliardo--Nirenberg inequality
     \[
        \| u \|_{L^\frac{2(d+2)}{d}} \le C \| u \|_{L^2}^\theta \| \nabla u \|_{L^2}^{1 - \theta}, \quad \text{with} \quad \frac{d}{2(d+2)} = \frac{\theta}{2} + (1 - \theta)\frac{d-2}{2d}, \quad \text{i.e.} \quad \theta = \frac{2}{d+2}.
     \]
    This gives
    \[
        \cF_1(u) \ge \| \nabla u \|_2^2 - M \int_{\R^d} | u |^{\frac{2(d+2)}{d}} \ge 
        \| \nabla u \|_2^2 \left( 1 - M C \lambda^{\frac{2}{d} }\right),
    \]
    which is positive if $\lambda < \frac{1}{MC}$. So $\lambda_c \ge \frac{1}{MC} > 0$.
    
    \medskip

    We now consider the case $\lambda = \lambda_c$. We follow the lines of~\cite{LewRot-20}. Consider a sequence $\lambda_n \to \lambda_c$ with $\lambda_n > \lambda_c$. Let $u_n$ be a positive radial decreasing minimizer for $\lambda_n$, so that $\cF_1(u_n) = J_1(\lambda_n)$ converges to $0$ as $n \to \infty$. Finally, let $\mu_n > 0$ be the corresponding Euler--Lagrange multiplier. Then, we must have $\Lambda(\mu_n) = \lambda_n < \lambda_1$. In particular, since $\Lambda(\cdot)$ is coercive by Proposition~\ref{prop:mu1_mu2}, there is $\mu_{\rm max} > 0$ so that $\Lambda(\mu) > \lambda_1$ for all $\mu > \mu_{\rm max}$. This proves that $0 < \mu_{\rm min} < \mu_n \le \mu_{\rm max} < \infty$, so the sequence $(\mu_n)$ is bounded, and, up to an undisplayed subsequence, we may assume that $\mu_n \to \mu_* \in [\mu_{\rm min}, \mu_{\rm max}]$. 
    
    \medskip
    
    The sequences $\cF_1(u_n)$ and $\| u_n \|_2$ are bounded. Using~\eqref{eq:F1_bounded_from_below}, this implies that the sequence $(u_n)$ is bounded in $H^1_{\rm rad}(\R^d)$. As in the proof of Lemma~\ref{lem:existence_minimizer_bosonic_lambda>lambda_c}, we deduce that $J_1(u_*) = 0$. It remains to prove that $u_* \neq 0$. Recall that
    \[
         - \Delta u_n + \mu_n u_n = -  u_n F_1'(u_n^2).
    \]
    Multiplying by $u_n$ and integrating gives
    \[
        \int_{\R^d} | \nabla u_n |^2 + \mu_n \lambda_n = - \int_{\R^d} | u_n |^2 F_1'(u_n^2), \quad \text{with} \quad
        t^2 F_1'(t^2) = \begin{cases}
          - q t^{2q} & \quad \text{if} \quad t \le 1 \\
          \fb t^2 - \fc r t^{2r} & \quad \text{if} \quad t \ge 1.  
        \end{cases}
    \]
    There is a constant $C > 0$ so that $t^2 | F_1'(t^2) | \le C t^{2r}$. Using a domination as in the proof of Lemma~\ref{lem:existence_minimizer_bosonic_lambda>lambda_c}, we have, up to a subsequence,
    \[
         0 > -\mu_* \lambda_c = -\lim_{n \to \infty} \mu_n \lambda_n \ge \lim_{n \to \infty}  \int_{\R^d} | u_n |^2 F_1'(u_n^2) =  \int_{\R^d} | u_* |^2 F_1'(u_*^2).
    \]
    So $u_* \neq 0$ as claimed. Let $\lambda' := \| u_* \|_2^2$. Then $0 < \lambda' \le \lambda_c$ be weak semi-continuity of the $L^2$ norm. Since $u_*$ is a minimizer for the $J_1(\lambda')$ problem, we deduce from the first point of the Lemma that $\lambda' = \lambda_c$. So $u_*$ is a non trivial minimizer for the $J_1(\lambda_c)$ problem. This optimizer satisfies $\cF_1(u_*) = 0$.
    \end{proof}


\section{The fermionic problem}
\label{sec:fermionic}

We now turn to the fermionic problem $I_\alpha(\lambda)$, which we recall is defined by
\[
    I_\alpha(\lambda) = \inf \left\{ \cE_\alpha(\gamma), \quad 0 \le \gamma \le 1, \quad \Tr(\gamma) = \lambda \right\}, \quad \text{with} \quad
    \cE_\alpha(\gamma) = \Tr( - \Delta \gamma) + \int_{\R^d} F_\alpha(\rho_\gamma).
\]
In the fermionic setting, we have the Pauli constraint $0 \le \gamma \le 1$. In particular, there is no equivalent to the scaling used in Lemma~\ref{lem:scaling_bosonic}. The problems $I_\alpha(\lambda)$ cannot be reduced to the study of the $\alpha = 1$ case, as we did in the bosonic case.

\subsection{Properties of $I_\alpha(\lambda)$}
\label{sec:fermions_Ialpha}

Since the functional $F_\alpha$ is concave, the map $\gamma \mapsto \cE_\alpha(\gamma)$ is concave as well. We are therefore in the setting studied in~\cite{Lew-11} (see also~\cite{GonLewNaz-21}). Before we go on, we record a classical result, although we were not able to find a reference for it. In what follows, we set, for $\lambda > 0$,
\begin{align*}
     \cP_\lambda := \left\{ \gamma \in \cS(L^2(\R^d)), \ 0 \le \gamma \le 1, \ \Tr(\gamma) = \lambda \right\}.
\end{align*}
For a convex set $K$, we denote by $\extreme(K)$ the set of extreme points of $K$, i.e. the set of points $p\in K$ such that, for all $x\neq y \in K$ and all $0 < t < 1$, we have $t x + (1 - t) y \neq p$. 

\begin{theorem}
    For all $\lambda > 0$, the set $\cP_\lambda$ is convex. In addition, writing $\lambda = N + t$ with $N \in \N$ and $t \in [0, 1)$, we have $\gamma \in \extreme(\cP_\lambda)$ iff $\gamma$ is of the form
    \begin{equation} \label{eq:extreme_gamma}
        \gamma  = \sum_{i=1}^{N} | \phi_i \rangle \langle \phi_i | + t | \phi_{N+1} \rangle \langle \phi_{N+1} |,
    \end{equation}
    with $(\phi_1, \cdots, \phi_{N+1})$ an orthonormal family in $L^2(\R^d)$. In particular, we have
    \[
        \cP_\lambda = (1-t) \cP_N + t \cP_{N+1}.
    \]
\end{theorem}

\begin{proof}
    Consider $\gamma \in \cP_\lambda$, and assume that $\gamma$ has at least two non integer eigenvalues $0 < \nu_1, \nu_2 < 1$. Then one can write $\gamma$ as the combination $\gamma = \frac12 (\gamma_\varepsilon + \gamma_{-\varepsilon})$ for $\varepsilon > 0$ small enough, where $\gamma_s$ denotes the operator with the same structure as $\gamma$ but with the eigenvalue shifts $\nu_1 \to \nu_1 + s$ and $\nu_2 \to \nu_2 - s$. So any such $\gamma$ is not an extreme point. The fact that operators of the form~\eqref{eq:extreme_gamma} are extreme points is a classical fact.
    
    \medskip
    
    For the second part, we first note that any extreme point of the form~\eqref{eq:extreme_gamma} can be written as 
    \[
    \gamma = (1 - t) \gamma_N + t \gamma_{N+1}, \quad \text{with} \quad \gamma_N = \sum_{i=1}^{N} | \phi_i \rangle \langle \phi_i | \in \cP_N, \quad \gamma_{N+1} = \gamma_N + | \phi_{N+1} \rangle \langle \phi_{N+1} | \in \cP_{N+1}.
    \]
    So $\extreme(\cP_\lambda) \subset  (1 - t) \cP_N + t \cP_{N+1}$. Taking the convex-hull on both sides, and using the Krein--Milman theorem shows that $\cP_\lambda \subset  (1 - t) \cP_N + t \cP_{N+1}$. The other inclusion is straightforward. Technically speaking, the Krein--Milman theorem only applies in finite dimensions, so one should first replace $L^2(\R^d)$ by $E \subset L^2(\R^d)$ of finite dimension, then take the limit $E \to L^2(\R^d)$. We leave the details to the reader (see also the argument in~\cite{Lie-81}).
\end{proof}

The following result is purely geometrical, and explains why quantization happens in concave problems. We record the following.

\begin{lemma}
    ~ 
    \begin{enumerate}
        \item \underline{Case $\lambda = N \in \N$}. If $I_\alpha(N)$ has a minimizer, then $I_\alpha(N)$ also has a minimizer which is a projector of rank $N$.
        \item \underline{Case $N < \lambda < N+1$}. If $\lambda$ is of the form $\lambda = N +t$ with $t \in (0, 1)$, then $\lambda = (1-t) N + t (N+1)$ and
        \begin{equation} \label{eq:concavity_between_N_and_N+1}
            I_\alpha(\lambda) \ge (1-t) I_\alpha(N) + t I_\alpha(N+1).
        \end{equation}
        In addition, if $I_\alpha(\lambda)$ has a minimizer $\gamma$, then $I_\alpha(\lambda)$ also has a minimizer of the form~\eqref{eq:extreme_gamma}.
    \end{enumerate}
\end{lemma}

The inequality~\eqref{eq:concavity_between_N_and_N+1} states that the map $\lambda \mapsto I_\alpha(\lambda)$ is piece-wise concave: it is concave on all intervals of the form $[N, N+1]$.

\begin{proof}
    We minimize the concave functional $\cE_\alpha$ on the convex set $\cP_\lambda$. If a minimizer exists in $\cP_\lambda$, then there is also a minimizer in $\extreme(\cP_\lambda)$, hence of the form~\eqref{eq:extreme_gamma}. This directly proves the result. We note that since $\cP_\lambda = (1-t) \cP_N + t \cP_{N+1}$, any $\gamma \in \cP_\lambda$ can be written of the form
    \[
        \gamma = (1 -t) \gamma_N + t \gamma_{N+1}, \quad \text{with} \quad \gamma_N \in \cP_N, \ \gamma_{N+1} \in \cP_{N+1}.
    \]
    In particular, the concavity of $\cE_\alpha$ implies that 
    \[
        \cE_\alpha(\gamma) \ge (1-t) \cE_\alpha(\gamma_N) + t \cE_\alpha(\gamma_{N+1}) \ge (1 - t) \cI_\alpha(N) + t \cI_\alpha(N+1).
    \]
    Taking the infimum on $\gamma \in \cP_\lambda$ shows the inequality~\eqref{eq:concavity_between_N_and_N+1}.
\end{proof}

\subsection{Existence of a minimizer}

We now study the existence of minimizers for the problem $I_\alpha(\lambda)$. In what follows, we focus on the case where $\lambda = N$ is an integer, and use the following result which is proved in~\cite{Lew-11}.
\begin{lemma} \label{lem:strong_binding_Lewin}
    Let $N \in \N$ with $N \ge 2$. If the following strong binding inequality holds
    \begin{equation} \label{eq:strong_binding_integers}
        \forall K \in \{1, 2, \cdots, N -1\},  \qquad I_\alpha(N) < I_\alpha(N-K)  + I_\alpha(K),
    \end{equation}
    then the problem $I_\alpha(N)$ has a minimizer.
\end{lemma}
The fact that strong binding inequality implies existence of minimizer is a standard technique in concentration compactness theory~\cite{Lio-84}. Let us emphasize that the usual inequality is over real numbers, that is
 \begin{equation} \label{eq:strong_binding_real}
    \forall \lambda \in (0, N), \qquad I_\alpha(N) < I_\alpha(N-\lambda)  + I_\alpha(\lambda).
\end{equation}
The fact that one can restrict to integers in~\eqref{eq:strong_binding_integers} comes from the inequality~\eqref{eq:concavity_between_N_and_N+1}. Indeed, consider any $\lambda \in (0, N)$, and write $\lambda = M + t$ with $M \in \{ 0, 1, \cdots, N-1\}$ and $t \in (0, 1)$, so that $N - \lambda = N -M - t  = (N - M - 1) + (1 - t)$. Inequality~\eqref{eq:concavity_between_N_and_N+1} shows that
\begin{align*}
    I_\alpha(\lambda) & \ge (1-t) I_\alpha(M) + t I_\alpha(M+1)\\
    I_\alpha(N - \lambda) & \ge  t I_\alpha(N-M-1) + (1 - t) I_\alpha(N - M)
\end{align*}
so if one assumes~\eqref{eq:strong_binding_integers}, we get, by summation,
\[
    I_\alpha(\lambda) + I_\alpha(N - \lambda) \ge (1 -t) \left[I_\alpha(M) + I_\alpha(N - M)\right] + t \left[I_\alpha(M+1) +   I_\alpha(N-M-1) \right] > I_\alpha(N).
\]
This proves that for concave functionals,~\eqref{eq:strong_binding_integers} implies~\eqref{eq:strong_binding_real}. Checking the strong binding inequalities over integers is enough to conclude  the existence of minimizers. Note that there are only a finite number of inequalities to check. This is one of the main reason why we have considered a concave functional at the very beginning.

\medskip

In particular, we note that $I_\alpha(N = 2)$ has a minimizer whenever
\begin{equation} \label{eq:strict_binding_I2}
    I_\alpha(2) < 2 I_\alpha(1),
\end{equation}
and there is only one inequality to check.

\medskip

Recall that $\alpha_c^I(\lambda)$ was defined in Lemma~\ref{lem:critical_alpha}. In what follows, we restrict our attention to integers, and we denote by $\alpha_c^{(N)} := \alpha_c^I(\lambda = N)$. It is the smallest value for which $I_\alpha(N) = 0$ for $\alpha < \alpha_c^{(N)}$, while $I_\alpha(N) < 0$ for $\alpha > \alpha_c^{(N)}$. According to Lemma~\ref{lem:critical_alpha}, we have $\alpha_c^{(1)} = \alpha_c^J(1) > 0$. In addition, since $N \mapsto I_\alpha(N)$ is non--increasing, we have 
\[
    \alpha_c^{(1)} \ge  \alpha_c^{(2)} \ge  \alpha_c^{(3)} \ge \cdots \ge  0.
\]

\begin{theorem} 
    \label{th:main_fermionic}
   ~
    \begin{itemize}
        \item The sequence $\alpha_c^{(N)}$ is bounded from below by a strictly positive constant, that is
        \[
            \alpha_c^{(\infty)} := \inf_N \alpha_c^{(N)}  = \lim_{N \to \infty} \alpha_c^{(N)} \quad \text{satisfies} \quad
            \alpha_c^{(\infty)} > 0.
        \]
        \item For all $\alpha > \alpha_c^{(\infty)}$, there is $N \in \N$ so that $I_\alpha(N)$ has a minimizer.
        \item Assume in addition $q < 2$. Then we have  $\alpha_c^{(2)} < \alpha_c^{(1)}$.
    \end{itemize}
\end{theorem}

The first statement implies that for $\alpha$ small enough, we have $I_\alpha( \lambda) = 0$ for all $\lambda > 0$. The system is never stable, no matter the number of fermions. In some sense, the self--interaction is not strong enough to counterpart the kinetic energy of the fermions.

As a corollary of the last statement, we deduce Theorem~\ref{th:main_fermionic_intro}, which we restate as follows
\begin{corollary}
    Assume $q < 2$, so that $\alpha_c^{(2)} < \alpha_c^{(1)}$. Then for all $\alpha \in (\alpha_c^{(2)} , \alpha_c^{(1)})$, the problem $I_\alpha(2)$ has a minimizer, while $I_\alpha(1)$ does not.
\end{corollary}

\begin{proof}[Proof of the Corollary and of Theorem~\ref{th:main_fermionic_intro}]
    Recall that for $N = 1$, the fermionic and bosonic are equal. Since $\alpha < \alpha_c^{(1)} = \alpha_c^J(1)$, we deduce that $I_\alpha(1) = J_\alpha(1)$ does not have a  minimizer (see Remark~\ref{rem:rescaled_existence_bosonic}). On the other hand, since $\alpha_c^{(2)} < \alpha < \alpha_c^{(1)}$, we have $I_\alpha(2) < 0 = I_\alpha(1)$, so the strict binding inequality~\eqref{eq:strict_binding_I2} is satisfied, and we deduce that the $N = 2$ problem has a minimizer by Lemma~\ref{lem:strong_binding_Lewin}.
\end{proof}

Note that since we have assumed $\frac{d+2}{d} < q$, the condition $q < 2$ can only happen in dimension $d \ge 3$.  We are not able to prove that $\alpha_c^{(\infty)} < \alpha_c^{(2)}$ however. We conjecture the following.

\begin{conjecture}
    If $q < 2$, then all inequalities are strict, that is 
    \[
        \alpha_c^{(1)} > \alpha_c^{(2)} > \alpha_c^{(3)} > \cdots >  \alpha_c^{(\infty)} > 0.
    \]
\end{conjecture}

This conjecture would imply that there are values of $\alpha$ for which the $N$--particle fermionic problem has a minimizer, while the $k$--particle problem has not, for all $k \in \{ 1, 2, \cdots, N-1\}$.

\begin{proof}[Proof of Theorem~\ref{th:main_fermionic}]
    For the first part, we use the Lieb-Thirring inequality~\cite{LieThi-75,LieThi-76}, which states that there is an optimal constant $K_d > 0$ so that
    \[
        \forall 0 \le \gamma = \gamma^* \le 1, \qquad \Tr( - \Delta \gamma ) \ge K_d \int_{\R^d} \rho_\gamma^{\frac{d+2}{d}}.
    \]
   Together with Lemma~\ref{lem:rescaling_beta}, this gives the lower bound (we set $\beta = \alpha^{q - \frac{d+2}{2}}$ and $\widetilde{\rho}(\bx) = \alpha \rho(\alpha^{1/d}\bx)$).
    \[
        \cE_\alpha(\gamma) = \alpha^{2/d} \widetilde{\cE}_\beta(\gamma) \geq \alpha^{2/d} \int_{\R^d} \left[ K_d \widetilde{\rho}^{\frac{d+2}{d}} + \beta F_1(\widetilde{\rho}) \right] \ge \inf_{\rho \ge 0}  \alpha^{2/d} \int_{\R^d} \left[ K_d {\rho}^{\frac{d+2}{d}} + \beta F_1({\rho}) \right].
    \]
    The function $f: t \mapsto \frac{| F_1(t)|}{t^{\frac{d+2}{d}}}$ is continuous, positive, and satisfies $\lim_{t \to 0} f(t) = \lim_{t\to \infty} f(t) = 0$, so $f$ is bounded from above by a constant $M > 0$. For $\beta < \frac{K_d}{M}$, the term in bracket is non--negative, and vanishes only for $\rho = 0$. This proves that $\cE_\alpha(\gamma) \ge 0$ for all $0 \le \gamma = \gamma^* \le 1$ when $\alpha$ is small enough.

    \medskip
    
    For the second point, we note that if $\alpha > \alpha_c^{(\infty)}$, then there is $N \in \N$ so that $\alpha_c^{(N-1)} \ge \alpha > \alpha_c^{(N)}$ (we set $\alpha_c^{(0)} = \infty$).  If $N \ge 2$, then $I_\alpha(1) = I_\alpha(2) = \cdots = I_\alpha(N-1) = 0$ while $I_\alpha(N) < 0$ so the strong binding inequality~\eqref{eq:strong_binding_integers} is satisfied, and $I_\alpha(N)$ has a minimizer. If $N = 1$, we recall that the bosonic and fermionic problems coincide, so we have $\alpha > \alpha_c^{(1)} = \alpha_c^J(1)$, and we deduce that $I_\alpha(1) = J_\alpha(1)$ has a minimizer (see Remark~\ref{rem:rescaled_existence_bosonic}).
    
    \medskip
    
    It remains to prove the strict inequality $\alpha_c^{(2)} < \alpha_c^{(1)}$ under the condition $q < 2$. The proof is similar to the one in~\cite{GonLewNaz-21, FraGonLew-21a, FraGonLew-21c}.  Consider $\alpha \ge \alpha_c^{(1)} = \alpha_c^J(1)$, and let $u \in L^2(\R^d)$ with $\| u \|_{L^2} = 1$ be the corresponding radial decreasing optimizer of the bosonic problem $\cF_\alpha(u)$. In what follows, we are mainly interested in the case $\alpha = \alpha_c^{(1)}$, where we have $\cF_\alpha(u) = 0$. For $R > 0$ a separation parameter, we set
    \[
        u^{(-)}_R(\bx) := u\left( \bx - \frac{R}{2} \be_1 \right), \qquad u^{(+)}_R(\bx) := u\left( \bx + \frac{R}{2} \be_1 \right) \quad \text{and} \quad
        \rho_R^{(\pm)} := \left| u_R^{(\pm)} \right|^2.
    \]
    Finally, we denote by $\gamma_R$ the orthogonal projector on ${\rm Ran} \{ u_R^{(-)}, u_R^{(+)} \}$, and by $\rho_R$ its density. Note that since $u_R^{(-)}$ is not proportional to $u_R^{(+)}$, $\gamma_R$ is a projector of rank $2$, hence a valid test operator for the $I_\alpha(N = 2)$ problem. We claim that, for $R$ large enough, we have 
    \begin{equation} \label{eq:I2<2I1}
        I_\alpha(2) \le \cE_\alpha(\gamma_R) <  2 I_\alpha(1).
    \end{equation}
    This will imply that the strong binding inequality~\eqref{eq:strong_binding_integers} holds for $N = 2$, hence that $I_\alpha(2)$ has a minimizer. In particular, for $\alpha = \alpha_c^{(1)}$, we get $I_\alpha(2) < 2 I_\alpha(1) = 0$, so $\alpha_c^{(2)} < \alpha_c^{(1)}$ as claimed.
    
    \medskip
    
    It remains to prove~\eqref{eq:I2<2I1}. One can write explicitly $\gamma_R$ using the Gram matrix
    \[
        G_R := \begin{pmatrix}
                1 & \varepsilon_R \\
                \varepsilon_R & 1
            \end{pmatrix}, \quad \text{with} \quad \varepsilon_R :=  \langle u_R^{(-)}, u_R^{(+)} \rangle.
    \]
    We have $\varepsilon_R \to 0$ (since the function $u$ is radial decreasing), so for $R$ large enough, the matrix $G_R$ is positive definite, and we can set 
    \[
        S_R := G_R^{-1/2} =  \begin{pmatrix}
            a_R & b_R \\
            b_R & a_R
        \end{pmatrix} = \begin{pmatrix}
            1 & - \frac{\varepsilon_R}{2} \\
            - \frac{\varepsilon_R}{2} & 1
        \end{pmatrix} + O(\varepsilon_R^2).
    \]
    Then, we have
    \begin{align*}
        \gamma_R 
        & = |  u_R^{(-)} \rangle \langle  u_R^{(-)} | +  |  u_R^{(+)} \rangle \langle  u_R^{(+)} | - \varepsilon_R \left( |  u_R^{(-)} \rangle \langle  u_R^{(+)} | + |  u_R^{(+)} \rangle \langle  u_R^{(-)} |   \right) +O(\varepsilon_R^2).
    \end{align*}
    and 
    \[
        \rho_R 
         = \left( \rho_R^{(-)} + \rho_R^{(+)} \right) - 2 \varepsilon_R  \sqrt{ \rho_R^{(-)}  \rho_R^{(+)}} + O_{L^1 \cap L^\infty}(\varepsilon_R^2).
    \]
  We obtain that
  \begin{align*}
    \cE_\alpha(\gamma_R) & = \Tr( - \Delta \gamma_R) + \int_{\R^d} F_\alpha(\rho_R) \\
    & = 2 I_\alpha(1) - 2 \varepsilon_R \int_{\R^d} \nabla  u_R^{(-)} \cdot \nabla  u_R^{(+)} +  \int_{\R^d} \left[ F_\alpha({\rho}_R) - F_\alpha(\rho_R^{(-)}) - F_\alpha(\rho_R^{(+)})\right] + O(\varepsilon_R^2).
  \end{align*}
  Now,~\eqref{eq:I2<2I1} follows from the following estimates. 

  \begin{lemma} \label{lem:estimates}
      There is $C \ge 0$ and $\mu>0$ so that
      \[
            \varepsilon_R \le C \re^{- \sqrt{| \mu |} R}, \quad \left| \int_{\R^d} \nabla  u_R^{(-)} \cdot \nabla  u_R^{(+)} \right| \le C \re^{- \sqrt{| \mu |} R},
      \]
      and there is $c > 0$ so that
      \begin{equation} \label{eq:estimate_integral_Falpha}
        \int_{\R^d} \left[ F_\alpha({\rho}_R) - F_\alpha(\rho_R^{(-)}) - F_\alpha(\rho_R^{(+)})\right] < - c R^{-q(d-1)} \re^{- q \sqrt{| \mu |} R} \qquad < 0.
      \end{equation}
  \end{lemma}
  In particular,  we obtain
  \[
    \cE_\alpha(\gamma_R) \leq 2 I_\alpha(1) - c R^{-q(d-1)}  \re^{- q \sqrt{| \mu |} R} + O \left(  \re^{- 2\sqrt{ | \mu |} R} \right).
\]
   In the case $q < 2$, the negative term is the leading order. We obtain as wanted that for $R$ large enough, we have $\cE_\alpha(\gamma_R) < 2 I_\alpha(1)$. 
\end{proof}
  
  We now give the proof of Lemma~\ref{lem:estimates}.
  \begin{proof}[Proof of Lemma~\ref{lem:estimates}]
  The function $u$ satisfies the Euler-Lagrange equation
      \[
      - \Delta u = - u F_\alpha'(u^2) - \mu u, \quad \text{for some}  \quad \mu  > 0.
      \]
  Note that the scaling here is different from the one used in the bosonic section, where we have only considered the case $\alpha = 1$.  According to (the rescaled version of) Proposition~\ref{prop:existence_ground_state}, there is $0 < c_- < c_+ < \infty$ so that, for all $\bx \in \R^d$, 
  \begin{equation} \label{eq:decay_u1}
      c_- \le \re^{\sqrt{\mu} | \bx |} (1 + | \bx |^{\frac{d-1}{2}}) u(\bx) \le c_+, 
      \quad  \re^{\sqrt{\mu} | \bx |} (1 + | \bx |^{\frac{d-1}{2}})  | \nabla u |(\bx) \le c_+.
  \end{equation}
   The first inequalities of Lemma~\ref{lem:estimates} follow directly from the upper bound for $u(\bx)$. One way to see this is to use the inequality $| u | \le C Y_{\sqrt{\mu}}$ where $Y_{\sqrt{\mu}}$ is the Yukawa potential, and $(Y_m*Y_m)(r)\sim rY_m(r)$ so that 
   \[
        | \varepsilon_R | \le \int_{\R^d}  | u_R^{(+)} | \cdot | u_R^{(-)} | = (| u | * | u |)(R) \le C \left( Y_{\sqrt{\mu}} * Y_{\sqrt{\mu}} \right)(R\be_1)\leq CE\frac{\re^{-R\sqrt{\mu}}}{1+R^{\frac{d-1}{2}}}\leq C \re^{-R\sqrt{\mu}}.
   \]
   
   Let us estimate the integral in~\eqref{eq:estimate_integral_Falpha}. Recall that $F_\alpha$ is a concave function with $F_\alpha(0) = 0$. In particular, it satisfies the inequality
   \[
    \forall a,b \in \R^+, \qquad F_\alpha(a + b) - F_\alpha(a) - F_\alpha(b) \le 0.
   \]
   The integrand in the LHS of~\eqref{eq:estimate_integral_Falpha} is therefore pointwise negative, and we can estimate its contribution by restricting the integral. 
   
   Since $u \to 0$ at infinity, there is $L > 0$ large enough so that $| u |^2(\bx) \le \frac{\alpha}{3}$ for all $| \bx | > L$. Then, there is $R_0 > 2L + 1$ such that, for all $R > R_0$ and all $\bx \in \cB(0, 1)$, we have $\rho_R^{(\pm)}(\bx) \le \frac{\alpha}{3}$,  and $\rho_R(\bx) \le \alpha$. We obtain
   \begin{align*}
    & \int_{\R^d} \left[ F_\alpha({\rho}_R) - F_\alpha(\rho_R^{(-)}) - F_\alpha(\rho_R^{(+)})\right]
     \le \int_{\cB(0, 1)} \left[ F_\alpha({\rho}_R) - F_\alpha(\rho_R^{(-)}) - F_\alpha(\rho_R^{(+)})\right] \\
     & \qquad = - \int_{\cB(0,1)} \left[\rho_R^q - (\rho_R^{(-)})^q - (\rho_R^{(+)})^q \right] \\
     & \qquad =  - \int_{\cB(0,1)} \left[(\rho_R^{(-)} + \rho_R^{(+)})^q - (\rho_R^{(-)})^q - (\rho_R^{(+)})^q \right] + O(\re^{-2 \sqrt{| \mu |} R}),
   \end{align*}
   where we have used estimates  as before for the last equality. The map $x,y \mapsto (x + y)^q - x^q - y^q$ is increasing both in $x$ and in $y$. In addition, the lower bound in~\eqref{eq:decay_u1} implies that there is $c > 0$ so that
   \[
    \forall \bx \in \cB(0, 1), \qquad \rho_R^{(\pm)}(\bx) \ge c \frac{\re^{- \sqrt{\mu} R}}{R^{(d-1)}}.
   \]
   We deduce that
   \begin{align*}
       - \int_{\cB(0,1)} \left[(\rho_R^{(-)} + \rho_R^{(+)})^q - (\rho_R^{(-)})^q - (\rho_R^{(+)})^q \right] \le - c^q (2^q - 2) | \cB(0, 1) |  \frac{\re^{- \sqrt{\mu} q R}}{ R^{q(d-1)}},
   \end{align*}
   which is strictly negative, since $q > 1$.
   
  \end{proof}

\subsection{Euler--Lagrange equations}
In this section, we focus on the case where the problem $I_\alpha(N)$ has a minimizer. We have provided general conditions for the existence of such minimizers in Theorem~\ref{th:main_fermionic}. We proved in particular that we must have $\alpha > \alpha_c^{(N)} \ge \alpha_c^{(\infty)} > 0$. 

Recall that the $k$-th min-max level $\mu_k$ of $H_*$ is the $k$-th lowest eigenvalue of $H_*$, counting multiplicities, if exists, and the bottom of the essential spectrum otherwise. 

\begin{lemma} 
    \label{lem:form_optimizer_fermionic}
    Assume that $I_\alpha(N)$ has a minimizer. Then, all minimizers $\gamma_*$ of $I_\alpha(N)$ are projectors, and satisfy the self--consistent equation
    \[
        \gamma_* = \1(H_* \le \mu_N), \quad \text{with} \quad H_* := - \Delta  + F_\alpha'(\rho_*).
    \]
    More specifically, the mean--field operator $H_*$ has at least $N$ negative eigenvalues satisfying
    \[
        -\mu_1 < -\mu_2 \le \cdots \le -\mu_N < -\mu_{N+1} \le 0. \quad \text{($\mu_k \ge 0$)}.
    \]
    If $(\phi_1, \cdots, \phi_N)$ is a corresponding orthonormal basis of eigenfunctions, then
    \[
        \gamma_* = \sum_{i=1}^N | \phi_i \rangle \langle \phi_i |, \quad \text{with} \quad
        \forall 1 \le i \le N, \qquad \left( - \Delta + F_\alpha' (\rho_*)\right) \phi_i = - \mu_i \phi_i.
    \]
\end{lemma}

Note the strict inequalities $-\mu_N < 0$ and $-\mu_N < -\mu_{N+1}$. This is reminiscent of Hartree--Fock theory, for which there is no {\em unfilled shell}~\cite{BacLieLos-94}, see also~\cite{GonLewNaz-21, FraGonLew-21c}. The proof below will show that it is a general fact when we consider concave functionals $F$ with $F'(0) = 0$.

\medskip

In what follows, we will use several times the following identity. For all operators $\gamma$ and $Q$ with corresponding densities $\rho$ and $\rho_Q$, we have, with $H_\gamma := - \Delta + F'_\alpha(\rho)$,
\begin{align}
	\label{eq:identity_gamma+Q}
    \cE_\alpha(\gamma + Q) & = \cE_\alpha(\gamma) + \Tr( - \Delta Q) +\int_{\R^d} F_\alpha (\rho + \rho_Q) - F_\alpha(\rho) \nonumber \\
   & = \cE_\alpha(\gamma) + \Tr \left[ H_\gamma Q\right] + \int_{\R^d} F_\alpha(\rho + \rho_Q) - F_\alpha(\rho) - F'_\alpha(\rho) \rho_Q \nonumber \\
    & = \cE_\alpha(\gamma) + \Tr \left[ H_\gamma Q\right] + \int_{\R^d} \left( \int_0^1 (1 - s) F_\alpha''(\rho + s \rho_Q) \rd s \right) \rho_Q^2. 
 \nonumber   \\
\end{align}
In particular, since $F_\alpha$ is concave, we have
\begin{equation} \label{eq:inequality_gamma+Q}
    \cE_\alpha(\gamma + Q)  \le \cE_\alpha(\gamma) + \Tr \left[ H_\gamma Q\right].
\end{equation}
We do not assume here that the operators are positive nor that they satisfy the Pauli principle for these expressions. We only require that $\rho \ge 0$ and $\rho + \rho_Q \ge 0$ in order for $F_\alpha(\rho)$ and $F_\alpha(\rho + s\rho_Q)$ to be well-defined for any $0\leq s \leq 1$.

\begin{proof}[Proof of Lemma~\ref{lem:form_optimizer_fermionic}]
    Let $\gamma_*$ be a minimizer. Then, for all $0 \le \gamma = \gamma^* \le 1$ with $\Tr(\gamma) = N$, and all $0 \le t \le 1$, the operator $\gamma_t := (1 - t) \gamma_* + t \gamma$ is a valid test operator for $I_\alpha(N)$. Using~\eqref{eq:inequality_gamma+Q} with $\gamma = \gamma_*$ and $Q = t \left[\gamma - \gamma_*\right]$, we get, with $H_* := - \Delta + F_\alpha'(\rho_*)$,
    \begin{align*}
       I_\alpha(N) & \le \cE_\alpha (\gamma_t) \le  I_\alpha(N) +  t \Tr \left[ H_* (\gamma - \gamma_*)\right] .
    \end{align*}
    We deduce that $\gamma_*$ is also an  optimizer of the linear problem
    \[
        \inf \left\{ \Tr \left( H_* \gamma \right)  , \qquad   0 \le \gamma = \gamma^* \le 1, \quad \Tr(\gamma) = N \right\}.
    \]
    In particular, since this linear problem has an optimizer $\gamma_*$, the mean--field operator $H_*$ has at least $N$ eigenvalues below its essential spectrum. Note that since $F_\alpha'(\rho_*)$ goes to $0$ at infinity, we have $\sigma_{\rm ess}(H_*) = [0, \infty)$. We denote by $-\mu_1 \le -\mu_2 \le \cdots \le -\mu_N \le 0$ these $N$ lowest eigenvalues, with multiplicities. Actually, all optimizers of the linear problem are of the form
    \[
    \gamma = \widetilde{\gamma} + \delta, \qquad \text{with} \quad \widetilde{\gamma} := \1(H_* < -\mu_N), \quad \text{and} \quad 0 \le \delta \le \1(H = -\mu_N) \quad \text{with} \quad
    \Tr(\delta) = N - \Tr (\widetilde{\gamma}).
    \]
    Using the min-max principle and  Perron--Frobenius, the first eigenvalue $-\mu_1$ is always simple, and we can choose a positive eigenvector $\phi_1 > 0$ for $\mu_1$. This already implies that $\rho_* \ge | \phi_1 |^2 > 0$. In addition, in the case $N > 1$, we have $\Tr(\widetilde{\gamma}) = M \in \N$ with $1 \le M < N$ and $\widetilde{\rho} \ge | \phi_1 |^2 > 0$ as well.
    
    In what follows, we write $\gamma_* = \widetilde{\gamma} + \delta_*$.  Note that $\Tr(\delta_*) = N - M > 0$.
    
    \medskip
    
    \underline{Let us prove that $-\mu_N < 0$.} Assume by contradiction that $\mu_N = 0$. For $t \in (0, 1)$ and $\lambda = N - t \in (N-1, N]$, consider the test state 
    \[
    \gamma_t := \widetilde{\gamma} + (1-t) \delta_* = \gamma_* - t \delta_*,
    \]
    which is a valid test operator for $I_\alpha(N- t(N-M))$. Using~\eqref{eq:identity_gamma+Q} with $\gamma = \gamma_*$ and $Q = - t \delta_*$, we get $I_\alpha(N - t(N-M)) \le \cE_\alpha (\gamma_t)$ with
    \begin{equation} \label{eq:inequ_muN}
        \cE_\alpha (\gamma_t) =I_\alpha(N) - t \underbrace{\left[ \Tr \left(H_* \delta_* \right) \right]}_{=-\mu_N(N-M) = 0}  + t^2 \int_{\R^d} \left( \int_0^1 (1 - s) F_\alpha''(\rho_{\gamma_*} - s t \rho_{\delta_*}) \rd s \right) \rho_{\delta_*}^2.
    \end{equation}
    For all $0 < s,t < 1$, we have $\rho_{\gamma_*} -s t \rho_{\delta_*} \ge \widetilde{\rho} > 0$ pointwise, so the function in parenthesis is pointwise strictly negative. Since $\rho_{\delta_*} > 0$ as well, we get the strict inequality
    \[
        I_\alpha(N - t(N-M)) < I_\alpha(N),
    \]
    which is a contradiction. This proves that $-\mu_N < 0$.
    
   \begin{remark}
       In the case $\alpha > \alpha_c^{(1)}$, we can use~\eqref{eq:inequality_gamma+Q} instead and get similarly that
       \[
        I_\alpha(N - t(N-M)) \le \cE_\alpha(\gamma_*) - t \left[ \Tr \left(H_* \delta_* \right) \right]  = I_\alpha(N) + t \mu_N (N-M).
       \]
       Taking $t = (N-M)^{-1}$ shows that $ I_\alpha(N - 1) \le I_\alpha(N) + \mu_N$. Together with the weak--binding inequality $I_\alpha(N) \le I_\alpha(N - 1) + I_\alpha(1)$, we deduce that
       \[
            -\mu_N \le I_\alpha(1) = J_\alpha(1) < 0.
       \]
       In other words, in the case $\alpha > \alpha_c^{(1)}$, the $N$-th lowest eigenvalue of the $N$-body mean--field operator is always bounded from above by a negative constant, independent of $N$.
   \end{remark}

    \underline{We now prove that $-\mu_N < -\mu_{N+1}$.}
        Let us denote by $\cM$ the set of all minimizers for the linear problem, that is
    \[
        \cM := \left\{ \widetilde{\gamma} + \delta, \qquad 0 \le \delta \le \1( H = -\mu_N), \quad \Tr(\delta) = N - M \right\}.
    \]
    Recall that we write $\gamma_* =  \widetilde{\gamma} + \delta_*$. Since $\gamma_*$ is a minimizer of $I_\alpha(N)$, it is also a minimizer among $\cM$. Using again~\eqref{eq:identity_gamma+Q}, with $\gamma = \gamma_*$ and $Q = \delta - \delta_*$, we deduce that for all $0 \le \delta \le \1(H = -\mu_N)$ with $\Tr(\delta) = N - M$, we have
    \[
        I_\alpha(N) \le \cE(\widetilde{\gamma} + \delta) = \cE(\gamma_* + Q) = \cE(\gamma_*) + \underbrace{\Tr \left[ H_\gamma (\delta - \delta_*)\right]}_{ = 0} + \int_{\R^d} \left( \int_0^1 (1 - t) F_\alpha''(\rho_* + t \rho_Q) \rd t\right) \rho_Q^2.
    \]
    In addition, since $0 < | \phi_1 |^2 \le \widetilde{\rho} \le \rho_* + t \rho_Q$ for all $t \in [0, 1]$, the function in parenthesis is pointwise negative by strict concavity of $F_\alpha''$. We deduce that for all $0 \le \delta \le \1(H = -\mu_N)$ with $\Tr(\delta) = N- M$, we have the pointwise equality
    \begin{equation} \label{eq:rho_delta}
        \forall \bx \in \R^d, \qquad \rho_\delta(\bx) = \rho_{\delta_*}(\bx).
    \end{equation}

    Since $-\mu_N < 0 = \inf \sigma_{\rm ess}(H_*)$, $-\mu_N$ is an eigenvalue of finite multiplicity. Let $K$ be the multiplicity of $-\mu_N$. Then, recalling that $1 \le M = \Tr(\1(H_* < -\mu_N)) < N$, we have
    \begin{equation} \label{eq:inequalities_muN}
    -\mu_1 < -\mu_2 \le \cdots \le -\mu_{M} < -\mu_{M+1} = -\mu_{M+2} = \cdots = -\mu_{M+K} < -\mu_{M+K+1}.
    \end{equation}
    with $-\mu_{M+1} = -\mu_{M+2} = \cdots = -\mu_{M+K} = -\mu_N$ (so $N \le M+K$). We denote by $(\psi_1, \cdots, \psi_K)$ an orthonormal family spanning ${\rm Ran} \,  \1(H_* =- \mu_N)$. For $1 \le \delta \le \1(H_* =- \mu_N)$ with $\Tr(\delta) = N - M$, we denote by $A$ its matrix in this basis, namely
    \[
        \delta = \sum_{\ell, k = 1}^K A_{k \ell} | \psi_k \rangle \langle \psi_\ell |, \quad \text{with} \quad A_{k \ell} := \langle \psi_k, \delta \psi_\ell \rangle, \quad 0 \le A = A^* \le 1, \quad \Tr(A) = N - M.
    \]
    We let $A_*$ be the matrix of $\delta_*$. Finally, for $\bx \in \R^d$, we denote by $M = M(\bx)$ the $K \times K$ matrix with coefficients
    \[
        M_{k \ell} = M_{k \ell} (\bx) := \psi_k(\bx) \overline{\psi}_\ell(\bx).
    \]
    In the basis $(\psi_1, \cdots, \psi_K)$, the equality~\eqref{eq:rho_delta} reads
    \[
        \forall A \in \cM_{K \times K} \quad \text{with} \quad 0 \le A = A^* \le 1, \quad \Tr(A) = N - M, \qquad \text{we have} \quad  \Tr (A M) = \Tr(A_* M).
    \]
    The conditions on $A$ shows that $N - M \le K$. We claim that we have equality. Indeed, minimizing and maximizing $A \mapsto \Tr(AM)$ among all matrices $A$ satisfying the previous conditions shows that the sum of the $N - M$ lowest eigenvalues of $M$ is equal to the sum of the $N - M$ highest eigenvalues of $M$. In the case $K > N - M$, this implies that all eigenvalues of $M$ are equal, so $M$ is proportional to the identity matrix, namely
    \[
        \forall 1 \le k, \ell \le K, \qquad  | \psi_k |^2(\bx) =   | \psi_\ell |^2(\bx) \quad \text{and} \quad \overline{\psi_k}(\bx) \psi_\ell(\bx) = 0.
    \]
    Taking the modulus of the last equality shows that $| \psi_k | (\bx) = 0$ pointwise, which is a contradiction.
    
    \medskip
    
    So we must have $K = N - M$. This already implies that $-\mu_N < -\mu_{N+1}$ according to~\eqref{eq:inequalities_muN}. In addition, since $K = N - M$, the only matrix $A$ satisfying $0 \le A \le 1$ with $\Tr(A) = K$ is the identity matrix. In other words, we have proved that
    \[
        \gamma_* = \1(H_* \le \mu_N) \quad \text{and} \quad \mu_N < \mu_{N+1} \le 0.
    \]
    This proves as wanted that $\gamma_*$ is a projector of rank $N$.
\end{proof}

We now turn to the decay properties of $\rho_*$.

\begin{lemma} \label{lem:decay_rho*}
If $\gamma_*$ is a minimizer of $I_\alpha(N)$, then, with the same notation as before, there are constants $0 < c_- < c_+ < \infty$ so that, for all $\bx \in \R^d$,
\[
 c_- \dfrac{\re^{- 2 \sqrt{ \mu_N } | \bx |}}{1 + | \bx |^{d-1}} \le \left[ \rho_* \right](\bx),
 \quad \text{and} \quad
 \rho_*(\bx) + | \nabla \rho_*(\bx) |  \le c_+ \dfrac{\re^{- 2 \sqrt{\mu_N} | \bx |}}{1 + | \bx |^{d-1}},
\]
where $[ \rho ](\bx) = \int_{{\rm SO}(n)} \rho(\cR \bx) \rd \omega(\cR)$ denotes the spherical average of $\rho$ over all rotations. 
\end{lemma}

     The proof is similar to~\cite[Lemma 19]{GonLewNaz-21} and uses the results of~\cite[Corollary 2.2]{HofHofSwe-85} on the spherical average of eigenvectors.  We skip it for the sake of brevity. For the upper bound on $\rho_*$, note that all eigenvalues $-\mu_1 < -\mu_2 \le \cdots <-\mu_N < 0$ are negative, hence all eigenvectors $\phi_i$ decays like $\re^{- \sqrt{ \mu_i } | \bx |} | \bx |^{-(\frac{d-1}{2})}$. In particular, $\phi_N$ decays slower than $\phi_k$ for $k \le N$. This explains why $\rho_*$ decays as $| \phi_N |^2$ at infinity, with exponential decay $\sim \re^{-2 \sqrt{\mu_N} | \bx |}$. We refer to~\cite[Lemma 19]{GonLewNaz-21} for the rest of the proof.
    
    \medskip

The general estimates for $\rho_*$ in Lemma~\ref{lem:decay_rho*} implies the following, which improves Theorem~\ref{th:main_fermionic}.

\begin{theorem}
    \label{th:existence_of_minimizers_for_infinitely_many_N}
    Assume $\alpha > \alpha_c^{(\infty)}$ and $q < 2$. For $N \in \NN$ such  that the problem $I_\alpha(N)$ has a minimizer, we have the strict inequality $I_\alpha(2N) < 2 I_\alpha(N)$. In particular, there is an infinity of integers $ N_1 < N_2 < N_3 < \cdots$ so that the problem $I_\alpha(N_k)$ has a minimizer.
\end{theorem}

The proof follows exactly the lines of~\cite[Theorem 4]{GonLewNaz-21}, so we do not repeat them here. We note that the proof of the strict binding inequality $I_\alpha(2N) < 2 I_\alpha(N)$ is similar to the one of  Theorem~\ref{th:main_fermionic}: we consider two copies of the minimizer $\gamma_N$ far away, and compute the interaction between these two copies. In the case $q < 2$, the interaction is attractive, giving the strict inequality $I_\alpha(2N) < 2 I_\alpha(N)$. The fact that it implies the existence of minimizers for all $N$ is a pigeon--hole argument, see~\cite[Theorem 4]{GonLewNaz-21}. 
 
\appendix

\section{Asymptotics of the bosonic ground state}
\label{sec:appendix:asymptotics}

The goal of this section is to study the properties of the ground state $u_\mu$ of~\eqref{eq:EL_mu} in the limits $\mu \to 0$ and $\mu \to \infty$. Parts of the proofs follow the lines of~\cite{MorMur-14,LewRot-20}. 

\medskip

Recall that $u_\mu$ is the {\bf unique} radial decreasing positive solution of the equation
\begin{equation} \label{eq:EL_mu_v2}
     - \Delta u = g_\mu(u), \quad \text{with} \quad g_\mu(t) = - t F_1'(t^2) - \mu t.
\end{equation}
We also recall that we set
\[
    \Lambda(\mu) := \| u_\mu \|_{L^2}^2, \quad \text{and} \quad \cF_1(u)  := \int_{\R^d} | \nabla u |^2 + \int_{\R^d} F_1(u^2).
\]
Finally, we recall the definition of $F_1$, given in~\eqref{eq:scaling_Falpha} by
\[
    F_1(t) := \begin{cases}
    - t^q & \quad t \le 1 \\
    - \fa + \fb t - \fc t^{r} & \quad t \ge 1.
\end{cases}, \quad \text{with} \quad 0 < r < \frac{d+2}{d} < q.
\]
The coefficients $\fa, \fb$ and $\fc$ are positive. We made this choice so that $F_1$ is of class $C^2$, concave, behaves as $-t^q$ around $0$, and as $-t^r$ at infinity.

\subsection{The limit $\mu \to 0$}
We first focus on the limit $\mu \to 0$.
It is useful to distinguish the case $q < \frac{d}{d-2}$ (sub-critical case) from the case $q \ge \frac{d}{d-2}$ (the critical or super--critical case). Note that the critical and the super--critical cases may only occur in dimension $d\ge 3$. 
\begin{lemma}[Limit $\mu \to 0$, subcritical case] ~Assume $q < \frac{d}{d-2}$. Let $v_0$ be the (unique) ground state to the NLS equation
    \begin{equation} \label{eq:NLS_v0}
        - \Delta v = - v + q v^{2q-1},
    \end{equation}
    and set $\mu_c := v_0(\bnull)^{-2(q-1)}$. Then, for all $0 < \mu < \mu_c$, we have
    \[
        u_\mu(\bx) = \mu^{\frac{1}{2(q-1)}} v_0(\sqrt{\mu} \bx).
    \]
    In particular, for $0 < \mu < \mu_c$, we have
    \begin{align*}
        \Lambda(\mu)  = \mu^{- \frac{d}{2(q-1)}(q - \frac{d+2}{d})} \| v_0 \|_2^2 \quad \xrightarrow[\mu \to 0]{} \infty, 
        \quad \text{and} \quad 
        \cF_1(u_\mu) = \mu^{\frac{q}{q-1} - \frac{d}{2}} \cF_1(v_0) > 0.
    \end{align*}
    
\end{lemma}

\begin{proof}
        The existence of a ground state solution for~\eqref{eq:NLS_v0} was first proved by Kwong~\cite{Kwo-89}, and follows from arguments similar to the one used in Proposition~\ref{prop:existence_ground_state}.  Now, we perform the scaling $u(\bx) = \alpha^{-1} v(\beta \bx)$, and find that $u$ solves~\eqref{eq:EL_mu} iff $v$ solves
    \begin{equation} \label{eq:scaling_EL_alpha_beta}
        - \Delta v + \frac{\mu}{\beta^2} v = \begin{cases}
            q \beta^{-2} \alpha^{-2(q-1)} v^{2q-1} & \quad \text{for} \quad v(\bx) \le \alpha \\
            - \frac{\fb}{\beta^2} v + \fc r \beta^{-2} \alpha^{-2(r-1)} v^{2r-1} & \quad \text{for} \quad v(\bx) > \alpha.
        \end{cases}
    \end{equation}
    Choosing $\beta^2 = \alpha^{-2(q-1)}=\mu $ gives
    \begin{equation} \label{eq:EL_v0}
        - \Delta v + v = \begin{cases}
            q v^{2q-1} & \quad \text{for} \quad v(\bx) \le \mu^{\frac{-1}{2(q-1)}} \\
            - \frac{\fb}{\mu} v + \fc r \mu^{\frac{r-q}{q-1}} v^{2r-1} & \quad \text{for} \quad v(\bx) > \mu^{\frac{-1}{2(q-1)}} .
        \end{cases}
    \end{equation}
    Since the function $v_0$ is radial decreasing, we have $v_0(\bx) \le v_0(\bnull)$. In the case $\mu < \mu_c$, this implies that $v_0(\bx) \le \mu^{\frac{-1}{2(q-1)}} $ pointwise, and the second condition is never satisfied. We deduce that $v_0$ is a ground state solution of~\eqref{eq:EL_v0}.  By uniqueness of the ground state, this implies that $u_\mu(\bx) = \mu^{\frac{1}{2(q-1)}} v_0(\sqrt{\mu} \bx)$. The scaling for $\Lambda(\mu)$ and $\cF_1(u_\mu)$ are straightforward (note that $F_1(u_\mu^2) = - u_\mu^{2q}$ for $\mu < \mu_c$). It remains to prove that $\cF_1(v_0) > 0$.
    
    \medskip
    
    We multiply~\eqref{eq:NLS_v0} by $v_0$ and integrate, and by $x \cdot \nabla v_0$ and integrate to obtain respectively
    \[
        \int_{\R^d} | \nabla v_0 |^2 = - \| v_0 \|^2_2 + q \int_{\R^d} | v_0 |^{2q}, 
        \quad \text{and} \quad
        \frac{d-2}{d} \int_{\R^d} | \nabla v_0 |^2 =  - \| v_0 \|^2_2 + \int_{\R^d} | v_0 |^{2q}.
    \]
    The second equality is the Pohozaev's identity. Taking the difference gives
    \[
        \frac{2}{d} \int_{\R^d} | \nabla v_0 |^2 = (q-1) \int_{\R^d} | v_0 |^{2q}.
    \]
    This leads to $\cF_1(v_0) = \| \nabla v_0 \|^2_2 \frac{1}{q-1}\left( q - \frac{d+2}{d}\right) > 0$.
\end{proof}

\medskip

The critical and supercritical case are more involved. We will use results by Berestycki and Lions from~\cite{BerLio-83, BerLio-83b} following~\cite{ColGlaMar-78}. We introduce the functions
\[
G_\mu(t) :=G_*(t) - \frac12 \mu t^2, \quad \text{with} \quad
G_*(t) := - \frac12 F_1(t^2) = \begin{cases}
    \frac12 t^{2q} & \quad t \le 1 \\
    \frac12 \fa - \frac{\fb}{2} t^2 + \frac{\fc}{2} t^{2r} & \quad t \ge 1,
\end{cases}
\]
so that $g_\mu = G_\mu'$. We also consider the corresponding optimization problems $M(G_\mu)$, where we define
\begin{equation} \label{eq:M(G)}
    M(G) := \sup \left\{ \int_{\R^d} G(v) , \quad v \in \dot{H}^1(\R^d), \ \| \nabla v \|_{L^2}^2 \le 1 \right\}.
\end{equation}
We recall that $\dot{H}^1(\R^d)$ is the homogenous Sobolev space of $\R^d$. It is the completion of $C^\infty_0(\R^d)$ for the norm $u \mapsto \| \nabla u \|_2^2$. Recall the Sobolev embedding that for $d \ge 3$, we have $\dot{H}^1(\R^d) \hookrightarrow L^{\frac{2d}{d-2}}(\R^d)$. Problem~\eqref{eq:M(G)} is a slight modification of the dual version of the problem studied in~\cite{BerLio-83, BerLio-83b}, see~\cite[Remark 3.2]{BerLio-83}. We prefer to work with this version, as the optimization space is independent of $G$. This will be more convenient when comparing the problems $M(G_\mu)$ for different values of $\mu$.

\medskip

We will use results from~\cite{BerLio-83, BerLio-83b}. For the sake of clarity, let us recall the main results that we will use. The first main result of~\cite{BerLio-83} is the following one.
\begin{lemma}[From \cite{BerLio-83}, Theorem 2]
    \label{lem:BL83_case1}
    Assume $d \ge 3$ , and that $g := G'$ satisfies the three following conditions:
  \begin{equation} \label{eq:condition_g}
        - \infty < \lim_{t \to 0} \frac{g(t) }{t} < 0, \quad  
        \lim_{t \to + \infty} \frac{g(t)}{t^\frac{d+2}{d-2}} \le 0, \quad
        \exists \zeta > 0 \quad  \text{s.t.} \quad G(\zeta) > 0.
    \end{equation}
    Then the problem $M(G)$ has an optimizer $v \in H^1(\R^d)$, which is positive radial decreasing. In addition, $v$ satisfies the Euler--Lagrange equation
    \[
        - \Delta v = \theta g(v), \quad \text{with} \quad \theta :=  \left( \frac{d-2}{2d} \right) \frac{1}{M(G)} > 0.
    \] 
    All such optimizers satisfy $\| \nabla v \|_{L^2} = 1$ and are of class $C^2$. Finally, all positive radial decreasing maximizing sequences are pre-compact in $H^1(\R^d)$, and converge to such an optimizer, up to a subsequence.
\end{lemma} 
We note that the condition $d \ge 3$ is always satisfied in the (super--)critical case. The expression of $\theta$ comes from the Pohozaev's identity. Finally, setting $u(\bx)  := v (\theta^{-1/2} \bx)$, we can check that $u$ is a solution to 
\[
    - \Delta u = g(u).
\]

\medskip

From~\cite[Section 5]{BerLio-83}, we also record the result corresponding to the <<zero-mass>> case. 
\begin{lemma}[From \cite{BerLio-83}, Theorem 4]
    \label{lem:BL83_case2}
    Assume $d \ge 3$ , and that $g := G'$ satisfies the three conditions
    \begin{equation} \label{eq:condition_g_2}
        g(0) = 0, \ \text{and} \ 
        \lim_{t \to 0} \frac{g(t)}{t^{\frac{d+2}{d-2}}} \le 0, \quad  
        \lim_{t \to + \infty} \frac{g(t)}{t^\frac{d+2}{d-2}} \le 0, \quad
        \exists \zeta > 0 \quad  \text{s.t.} \quad G(\zeta) > 0.
    \end{equation}
    Then the results of Lemma~\ref{lem:BL83_case1} holds, upon replacing $H^1(\R^d)$ by $\dot{H}^1(\R^d)$.
\end{lemma}

The differences with the previous case Lemma~\ref{lem:BL83_case1} is that $g(t)/t$ is now allowed to vanish at $0$ (hence the term <<zero mass>>). However, the corresponding solutions may not belong to $L^2(\R^d)$.

\medskip

In our case with $g = g_\mu$, we can check that the conditions~\eqref{eq:condition_g} are satisfied, as $\lim_{t \to 0} g_\mu(t)/t = - \mu < 0$, $\lim_{t \to \infty} g_\mu(t)/t^{\frac{d+2}{d-2}} = 0$ and $G_\mu(t) > 0$ for $t$ large enough. Finally, if $v_\mu$ an optimizer of $M(G_\mu)$, then $u(\bx) := v_\mu(\theta_\mu^{-1/2} \bx)$ is a positive solution of $- \Delta u = g_\mu(u)$. However, we proved in Proposition~\ref{prop:existence_ground_state} that this last PDE has a unique positive solution $u_\mu$. This proves that $u = u_\mu$, and that the problem $M(G_\mu)$ has a unique radial non--increasing optimizer $v_\mu$. The functions $v_\mu$ and $u_\mu$ are linked by the relations
\begin{equation} \label{eq:relation_u_v}
    u_\mu(\bx) := v_\mu (\theta_\mu^{-1/2} \bx), \qquad \theta_\mu :=  \left( \frac{d-2}{2d} \right) \frac{1}{M(G_\mu)} > 0.
\end{equation}

We are now in position to prove the result in the super--critical case.

\begin{lemma}[Limit $\mu \to 0$, super--critical case] Assume $q > \frac{d}{d-2}$. Consider the ``zero-mass'' equation
    \begin{equation}\label{eq:zero-mass}
        -\Delta u = g_*(u), \qquad g_*(t) = G_*'(t) = \begin{cases}
            q t^{2q-1} & \quad t \le 1\\
            -\fb t + \fc r t^{2r-1} & \quad t \ge 1.
        \end{cases}
    \end{equation}
    Then $u_\mu$ converges to the unique positive radial decreasing solution $u_*$ of~\eqref{eq:zero-mass}, strongly in $\dot{H^1}\cap L^{\frac{2d}{d-2}}(\RR^d)$.  In addition, we have
    \begin{equation} \label{eq:F1_u_mu_supercritical}
        \cF_1(u_\mu) \xrightarrow[\mu \to 0]{}  \frac{2}{d}  \left(\frac{d-2}{2d}\frac{1}{M(G_*)}\right)^{\frac{d-2}{d}} > 0.
    \end{equation}
 \end{lemma}
 
 \begin{proof}
     Note that the function $g_*$ satisfies the hypothesis of Lemma~\ref{lem:BL83_case2} (since $q > \frac{d}{d-2}$), but {\em does not} satisfy the ones of Lemma~\ref{lem:BL83_case1}, as $\lim g(t)/t = 0$. Lemma~\ref{lem:BL83_case2} then shows that $u_*$ is of class $C^2$. In addition, using~\cite[Theorem 5]{FluMul-98}, we have that $\lim_{|x|\to \infty}|x|^{d-2}u_*(x)$ exists and is a positive finite number. As a consequence, we can use~\cite[Theorem 2]{Tan-01} to prove the uniqueness of positive solution of~\eqref{eq:zero-mass}. So the problem $M(G_*)$ has a unique optimizer, that we denote by $v_*$. The functions $u_*$ and $v_*$ are linked by a relation similar to~\eqref{eq:relation_u_v}. Note that the decay of $u_*$ (hence $v_*$) implies that $u_*$ and $v_*$ are in $L^2(\R^d)$ only in dimensions $d \ge 5$. 
     
     \medskip
     
    For any $\mu>0$, let $v_\mu \in H^1(\R^d)$ be the unique positive radial decreasing optimizer of $M(G_\mu)$, so that
    \[
        \int_{\R^d}G_\mu(v_\mu)=M(G_\mu)>0, \quad \| \nabla v_\mu \|_{L^2}^2 = 1.
    \]
   Our goal is to show that, for any sequence $(\mu_n)_n$ such that $\lim_{n\to\infty}\mu_n=0$, $(v_{\mu_n})_{n}$ is a maximizing sequence for $M(G_*)$. Let us first show that
    \begin{equation}
        \label{eq:lim_MGmu_0}
        \lim_{n\to \infty}M(G_{\mu_n})=M(G_*).
    \end{equation}
    First, since $v_{\mu_n}\in H^1(\R^d)$, we have
    \begin{align}\label{eq:lim_MGmu_0_upper}
        M(G_*)\ge \int_{\R^d} G_{*}(v_{\mu_n})= \int_{\R^d} G_{\mu_n}(v_{\mu_n})+\frac{1}{2}\mu_n \int_{\R^d} |v_{\mu_n}|^2\ge M(G_{\mu_n}).
    \end{align}
    For the converse inequality, let us first consider the case $d \ge 5$. In this case, we have $v_*(\bx )\simeq_{\bx \to \ii} \frac{1}{\av{\bx }^{d-2}}$, so that $v_*\in L^2(\R^d)$, and
    \begin{align*}
        M(G_{\mu_n})\ge \int_{\R^d} G_{\mu_n}(v_{*})= \int_{\R^d} G_{*}(v_{*})-\frac{1}{2}\mu_n \int_{\R^d} |v_{*}|^2= M(G_{*})(1+o(1)).
    \end{align*}
    This proves~\eqref{eq:lim_MGmu_0} in dimensions $d \ge 5$. For $d\in \set{3,4}$, the function  $v_*(\bx )$ is not in $L^2(\RR^d)$. We use a cutoff function and proceed similarly. Let $R>0$ and let $0\leq \chi_R\leq 1$ a $C^\ii_c$ a cutoff function such that $\chi(s )=1$ for $\av{s }\leq R$, $\chi(s )=0$ for $\av{s }\geq 2R$ and $\av{\chi'_R}\leq 2/R$. Denote by $v_R=\chi_Rv_*$. A straightforward computation gives, for $R$ large enough,
    \begin{align*}
        &\int_{\R^d} |\nabla v_R|^2=\int_{\R^d} |\nabla v_*|^2 + O(R^{-(d-2)})\\
        &\int_{\R^d} G_*(v_R)=\int_{\R^d} G_*(v_*) + O(R^{d-2q(d-2)})\\
        &\int_{\R^d} |v_R|^2= 
        \begin{cases}
            O(R) & d=3\\
            O(\log(R)) & d=4.
        \end{cases}
    \end{align*}
  Let $\delta_R :=\|\nabla v_R\|_{L^2}^2$ and define $\tilde v_R(\bx)=v_R(\delta_R^{1/(d-2)}\bx)$,  so that $\|\nabla\tilde v_R\|_{L^2}^2=1$ and $\tilde v_R(\bx)$ is admissible  for the optimization problem $M(G_{\mu_n})$ for any $n\in \N$. We have
    \begin{align*}
        M(G_{\mu_n})\ge \int_{\R^d} G_{\mu_n}(\tilde v_R)=\delta_R^{-\frac{d}{d-2}}\left(\int_{\R^d} G_{*}(v_R)-\frac{\mu_n}{2}\int_{\R^d}v_R^2\right).
    \end{align*}
    Choosing $R = R_n=\mu_n^{-1/2}$ gives
    \begin{align*}
        \int_{\R^d} G_{*}(v_{R_n})-\frac{\mu_n}{2}\int_{\R^d}v_{R_n}^2&=\int_{\R^d} G_*(v_*)\left(1+O(\mu_n^{(2q(d-2)-d)/2})+ \begin{cases}
            O(\mu_n^{1/2}) & d=3\\
            O(\mu_n\log(\mu_n)) & d=4
        \end{cases}\right)\\
        &= M(G_*)(1+o(1))
    \end{align*} 
    and 
    \begin{align*}
        \delta_{R_n}^{-\frac{d}{d-2}}=\|\nabla v_*\|_{L^2}^{-\frac{2d}{d-2}}(1+O(\mu_n^{(d-2)/2}))^{-\frac{2d}{d-2}}\ge (1+o(1)).
    \end{align*}
    As a conclusion, $M(G_{\mu_n})\ge M(G_*)(1+o(1))$, which concludes the proof of~\eqref{eq:lim_MGmu_0}. 
    
    \medskip
    
    Now,~\eqref{eq:lim_MGmu_0}, together with~\eqref{eq:lim_MGmu_0_upper}, implies
    \begin{equation}
        \label{eq:convL2supercritical}
        \lim_{n\to +\infty}\mu_n\|v_{\mu_n}\|_{L^2}^2=0.
    \end{equation}
    As a consequence, 
    \begin{equation*}
        \int_{\R^d} G_*(v_{\mu_n})=M(G_{\mu_n})+\frac{\mu_n}{2}\int_{\R^d}v_{\mu_n}^2\to M(G_*)
    \end{equation*}
    and $(v_{\mu_n})_n$ is a maximizing sequence for $M(G_*)$. Together with Lemma~\ref{lem:BL83_case2}, $(v_{\mu_n})_n$ converges in $\dot{H^1}(\R^d)$ to $v$, an optimizer of $M(G_*)$, up to a subsequence. Since the optimizer is unique, we deduce that the full sequence $(v_{\mu_n})_n$ converges to $v_*$ in $\dot{H^1}$. To conclude the proof, we recall that $u_\mu(\bx) = v_\mu(\theta_\mu^{-1/2} \bx)$ with $\theta_\mu = \left( \frac{d-2}{2d} \right) \frac{1}{M(G_\mu)}$ and $u_*(\bx) = v_*(\theta_*^{-1/2} \bx)$ with  $\theta_* = \left( \frac{d-2}{2d} \right) \frac{1}{M(G_*)}$. So the convergence $M(G_\mu)  \to M(G_*)$ implies the convergence $u_\mu$ to $u_*$ in $\dot{H}^1(\R^d)$.
    
    \medskip
    
    It remains to prove~\eqref{eq:F1_u_mu_supercritical}, namely that $\cF_1(u_\mu) > 0$ for $\mu$ small enough. Recall the scaling~\eqref{eq:relation_u_v}, which implies
    \[
         \int_{\R^d} G_\mu(u_\mu) = \theta_\mu^{\frac{d}{2}} M(G_\mu) \xrightarrow[\mu \to 0]{} \theta_*^\frac{d}{2}M(G_*), 
         \quad \text{and} \quad \mu \int_{\R^d} | u_\mu |^2 = \mu \theta_\mu^{\frac{d}{2}} \int_{\R^d} | v_\mu |^2  \xrightarrow[\mu \to 0]{} 0,
    \]
    where we used~\eqref{eq:convL2supercritical} in the last limit. On the other hand, since $- \Delta u_\mu = G_\mu'(u_\mu)$, we have the Pohozaev's identity
     \[
        \| \nabla u_\mu \|_{L^2}^2 = \frac{2 d}{d-2} \int_{\R^d} G_\mu(u_\mu).
    \]
    Together with the fact that $F_1(u^2) = - 2 G_\mu(u) -  \mu u^2$, we get
    \begin{align*}
        \cF_1(u_\mu) & 
        =  \| \nabla u_\mu \|^2 - 2 \int_{\R^d} G_\mu(u_\mu) - \mu \int_{\R^d} | u_\mu |^2 
        \xrightarrow[\mu \to 0]{}  \frac{2}{d}  \left( \frac{d-2}{2d} \frac{1}{M(G_*)} \right)^{\frac{d-2}{2}},
    \end{align*}
    which concludes the proof.
    \end{proof}

    It remains to prove the result in the critical case $q = \frac{d}{d-2}$. 

    \begin{lemma}[Limit $\mu \to 0$, critical case] Assume $q = \frac{d}{d-2}$. There exists a sequence $(\lambda_\mu)_{\mu}$, $\lambda_\mu\in (0,\infty)$ such that the rescaled function 
    \begin{equation*}
        \lambda_{\mu}^{\frac{d-2}{d}}u_{\mu}(\lambda_\mu\cdot)
    \end{equation*}
    converges strongly in $\dot{H^1}\cap L^{\frac{2d}{d-2}}(\RR^d)$, to the
    Sobolev optimizer 
    \begin{equation}\label{eq:AubinTalenti}
        S(\bx)=\left(1+\frac{|\bx|^2}{(d-2)^2}\right)^{-\frac{d-2}{2}}.
    \end{equation}
     In addition, for $\mu$ small enough, we have $\cF_1(u_\mu) > 0$.
\end{lemma}

The function $S(\cdot)$ is also the unique positive radial-decreasing solution (up to dilations) the ``Emden–
Fowler'' equation
\begin{equation}\label{eq:emden-fowler}
    -\Delta S =\frac{d}{d-2} S^{\frac{d+2}{d-2}}.
\end{equation}
Since $S\le 1$ pointwise, it is also a solution to the zero--mass equation~\eqref{eq:zero-mass} when $q = \frac{d}{d-2}$.

\begin{proof}
    Let $q = \frac{d}{d-2}$, we introduce the modified function $\widetilde{G}_*(t):=\frac{1}{2}t^{2q} = \frac{1}{2} t^{\frac{2d}{d-2}}$ for all $t\ge 0$, and we consider the corresponding optimization problem $M(\widetilde G_*)$ defined by \eqref{eq:M(G)}. Note that the function $G_*$ does not satisfy the hypothesis of Lemma~\ref{lem:BL83_case1} nor the ones of Lemma~\ref{lem:BL83_case2}, since $\widetilde{g}_* := \widetilde{G}_*'$ satisfies 
    \[
        \lim_{t \to 0} \widetilde{G}_*(t) t^{\frac{d+2}{d-2}} = \frac{d}{d-2} > 0.
    \]
    In this case, we resort to the results of~\cite{Lio-85a} (which deals with the critical case). It follows from \cite[Theorem I.1]{Lio-85a} that all radial decreasing maximizing sequences for $M(\tilde G_*)$ are pre-compact in $\dot{H^1}(\R^d)$ up to a dilation. As a consequence, $M(\widetilde G_*)$ admits an optimizer $\widetilde v^*$ which satisfies the equation 
    \begin{equation*}
        -\Delta \widetilde v_*=\widetilde \theta_* \frac{d}{d-2}\tilde v_*^{\frac{d+2}{d-2}}\quad \text{ with }\quad  \widetilde \theta_*=\frac{d-2}{2d}\frac{1}{M(\widetilde G_*)}.
    \end{equation*}
    More precisely, any optimizer of $M(\widetilde G_*)$ is of the form $\widetilde v_{*,\lambda}(\bx)=\lambda^{\frac{d-2}{2}}\widetilde v_{*,1}(\lambda\bx)$ with $\widetilde v_{*,1}(\bx)= S({\widetilde \theta_*}^{1/2}\bx)$. Here $S$ is the Aubin--Talenti function defined by~\eqref{eq:AubinTalenti}. Since $S\le 1$ pointwise, the function $S$ is a solution to the zero--mass equation~\eqref{eq:zero-mass}. This is also true for $\tilde v_{*,\lambda}$ whenever $\lambda\le 1$.

    We claim that $M(\widetilde G_*)=M(G_*)$. Indeed, on the one hand, since $G_*\le \widetilde G_*$, we have $M(\widetilde G_*)\ge M(G_*)$. On the other hand, let $\widetilde v_*$ be an optimizer of $M(\widetilde G_*)$. Using the dilation described above, $\widetilde v_*$ can be chosen such that $|\widetilde v_*|\le 1$. As a consequence, $\widetilde G_*(\widetilde v_*)=G_*(\widetilde v_*)$ and 
    \begin{equation*}
        M(G_*)\ge \int_{\R^d} G_*(\widetilde v_*)= \int_{\R^d} \widetilde G_*(\widetilde v_*)= M(\widetilde G_*).
    \end{equation*}
    
    For any $\mu>0$, let $v_\mu$ be the unique positive radial decreasing optimizer of $M(G_\mu)$. We have $v_\mu\in H^1(\R^d)$, $\int_{\R^d}G_\mu(v_\mu)=M(G_\mu)>0$ and $\| \nabla v_\mu \|_{L^2}^2 = 1$. Our goal is to show that, for any sequence $(\mu_n)_n$ such that $\lim_{n\to\infty}\mu_n=0$, $(v_{\mu_n})_{n}$ is a maximizing sequence for $M(\widetilde G_*)$.

    As $\tilde{v}_*$ is an optimizer of $M(G_*)$, we can reason as in the proof of~\eqref{eq:lim_MGmu_0} to obtain 
    \begin{equation}\label{eq:lim_MGmu_star}
        \lim_{n\to\infty}M(G_{\mu_n})=M( G_*)= M(\tilde G_*).
    \end{equation}
    Moreover, for any $n\in \N$,
    \begin{align*}
        0\le \frac{\mu_n}{2}\int_{\R^d}v_{\mu_n}^2=\int_{\R^d} G_*(v_{\mu_n})-\int_{\R^d} G_{\mu_n}(v_{\mu_n})\le \int_{\R^d} \widetilde G_*(v_{\mu_n})-M(G_{\mu_n})\le M(\widetilde G_*)-M(G_{\mu_n}),
    \end{align*}
    As a consequence,
    \begin{equation}
        \label{eq:convL2critical}
        \lim_{n\to +\infty}\mu_n\|v_{\mu_n}\|_{L^2}^2=0.
    \end{equation}
    and
    \begin{align*}
        0\le M(\widetilde G_*)-\int_{\R^d}\widetilde G_*(v_{\mu_n})\le M(\widetilde G_*)-\int_{\R^d} G_*(v_{\mu_n})=M(\widetilde G_*)-M(G_{\mu_n})-\frac{\mu_n}{2}\int_{\R^d}v_{\mu_n}^2\xrightarrow[n \to \infty] {} 0.
    \end{align*}
    Hence, $(v_{\mu_n})_n$ is a maximizing sequence for $M(G_*)$. It follows from~\cite[Theorem I.1]{Lio-85a} that there exists a sequence $(\lambda_n)_n$ such that the rescaled sequence $(\lambda_n^{\frac{d-2}{d}}\tilde v_{\mu_n}(\lambda_n\cdot))_n$ converges strongly in $\dot{H^1}(\R^d)$ to $\tilde v_{*,1}$, up to an undisplayed subsequence. Since $u_\mu(\bx) = v_\mu(\theta_\mu^{-1/2} \bx)$ with $\theta_\mu = \left( \frac{d-2}{2d} \right) \frac{1}{M(G_\mu)}$ and $S(\bx) = \tilde v_{*,1}(\tilde \theta_*^{-1/2} \bx)$, we conclude that the rescaled sequence
    \begin{equation*}
        \lambda_n^{\frac{d-2}{d}}u_{\mu_n}(\lambda_n\cdot)
    \end{equation*}
    converges to $S$ strongly in $\dot{H^1}(\R^d)$.  The proof $\cF_1(u_\mu) > 0$ for $\mu$ small enough follows the same lines as the proof of~\eqref{eq:F1_u_mu_supercritical}.
\end{proof}

\subsection{The limit $\mu \to \infty$}

The behaviour of $u_\mu$ in the limit $\mu \to \infty$ is more involved. 
One reason is that the functions 
\begin{equation*}
    g_\mu(t)=-\mu t+\begin{cases}
        qt^{2q-1}&t\le 1\\
        - \fb t + \fc r  t^{2r-1} & t\ge 1
    \end{cases}
\end{equation*}
diverge pointwise to $- \infty$ as $\mu \to \infty$. One way to handle this divergence is to make a change of variable. We  introduce
\[
     \widetilde{g}_\mu(t) := A g_\mu( B t),
\]
where the positive parameters $A = A(\mu)$ and $B = B(\mu)$ will be chosen below.
It is easy to check that $u_\mu$ is a solution to 
$-\Delta u=g_\mu(u)$
iff $\widetilde{u}_\mu$ is a solution to  $-\Delta u= \widetilde{g}_\mu(u)$, with
\[
    \widetilde{u}_\mu(\bx) = B^{-1} u_\mu \left( \sqrt{AB} \bx\right) .
\]
Let us now choose the parameters $A$ and $B$ that will serve our purposes. We have
\begin{align*}
    \widetilde{g}_\mu(t)&=-AB \mu  t+\begin{cases}
        A B^{2q-1}q t^{2q-1}&Bt\le 1\\
        - AB\fb t + AB^{2r-1}\fc r  t^{2r-1} & Bt\ge 1
    \end{cases}\\
    &=-AB (\mu+\fb)  t+ AB^{2r-1}\fc r  t^{2r-1}+\begin{cases}
        AB\fb t+ A B^{2q-1}q t^{2q-1}- AB^{2r-1}\fc r  t^{2r-1}&Bt\le 1\\
        0 & Bt\ge 1.
    \end{cases}
\end{align*}
We make the following choice
\[
    \begin{cases}
        A & =  (\fb + \mu)^{-\frac{2r-1}{2(r-1)}}\\
        B & = (\fb + \mu)^{\frac{1}{2(r-1)}}
    \end{cases} , \quad \text{so that} \quad
    \begin{cases}
     AB \left({\fb} +  \mu \right) & = 1 \\
    A B^{2r-1} & = 1
    \end{cases}
\]
In particular, we have $B \to \infty$ as $\mu \to \infty$, so the condition $Bt \le 1$ is satisfied on a smaller and smaller interval as $\mu$ increases. Actually, with these values, the function $\widetilde{g}_\mu$ simplifies into
\begin{align*}
    \widetilde{g}_\mu(t)=- t+ \fc r  t^{2r-1}+\begin{cases}
        \frac{\fb}{\fb +\mu} t+ (\fb +\mu)^{\frac{q-r}{r-1}}q t^{2q-1}-\fc r  t^{2r-1}&t\le (\fb + \mu)^{-\frac{1}{2(r-1)}}\\
        0 & t\ge (\fb + \mu)^{-\frac{1}{2(r-1)}}
    \end{cases}
\end{align*}

We have the following result.
\begin{lemma}\label{lem:asymp_infinity}
    Let $Q$ be the (unique) ground state to the equation
    \begin{equation}\label{eq:semi-linear-v0}
        - \Delta w = - w + \fc r w^{2r-1}.
    \end{equation}
    For $\mu > 0$ (large), we introduce the rescaled function $\widetilde{u}_\mu(\bx) := (\mu+\fb)^{-\frac{1}{2(r-1)}} u_\mu( \bx/\sqrt{\mu+\fb})$. Then, we have, in the limit $\mu \to \infty$,
     \[
        \left\| \widetilde{u}_\mu -  Q \right\|_{H^2} = o(1).
    \]
    In particular, we have
    \[
         \Lambda(\mu) = \| u_\mu \|_2^2 =  (\mu+\fb)^{\frac{d}{2(r-1)}( \frac{d+2}{d} - r)} \| Q \|_2^2 (1 + o(1)) \xrightarrow[\mu \to \infty]{} \infty.
    \]
\end{lemma}

    For any $\mu>0$, let $\alpha = \alpha_\mu=(\fb + \mu)^{-\frac{1}{2(r-1)}}$. As explained above, the rescaled function $\widetilde{u}_\mu$ is the (unique) ground state of
    \begin{equation} \label{eq:IFT_NLS}
        - \Delta \widetilde{u} + \widetilde{u}=  \fc r \widetilde{u}^{2r-1} +  E(\alpha,\widetilde{u}),
    \end{equation}
    with the function
    \[
    E(\alpha,t) := \begin{cases}
        {\fb}\alpha^{2(r-1)} t +q \alpha^{-2(q-r)} t^{2q-1}- \fc r t^{2r-1} & \quad \text{if} \quad t < \alpha; \\
        0 & \quad \text{if} \quad t \ge \alpha.
        \end{cases}
    \]
    Note that $\alpha_\mu\to 0$ as $\mu\to \infty$ and   
$$
\Lambda(\mu)=\int\av{u_\mu}^2=\bra{\fb+\mu}^{\frac{d}{2(r-1)}\bra{\frac{d+2}{d}-r}}\norm{\widetilde{u}_\mu}^2_2= \frac{1}{\alpha_\mu^{d\bra{\frac{d+2}{d}-r}}}\norm{\widetilde{u}_\mu}^2_2. 
$$
Let us record some properties of the function $E$.
\begin{proposition}
    \label{prop:E_alpha}
    For all $\alpha > 0$, the function $E(\alpha,\cdot )$ is $C^1$, positive on $(0,\alpha)$, and satisfies the inequality
  \begin{align*}
        &0\le {E(\alpha,t)} 
        \leq (\fb+q)\alpha^{2(r-1)}t\quad \text{for all }t\in (0,\alpha),\\
        &\sup_t\av{\partial_tE(\alpha,t)} \le \bra{\fb+q(2q-1)+\fc r}\alpha^{2(r-1)}.
    \end{align*}
\end{proposition}
    \begin{proof}
    Differentiating $E$ w.r.t. to $t$ for $0 < t < \alpha$, we have
    \begin{align*}
        \partial_t E(\alpha,t)  & = {\fb}\alpha^{2(r-1)}+ 	 q(2q-1) \alpha^{-2(q-r)}t^{2(q-1)} - \fc r (2r-1)t^{2(r-1)} \\
        \partial_t^2 E(\alpha,t) & = 2t^{2r-3} \left(  q(2q-1)(q-1)  \alpha^{-2(q-r)} t^{2(q-r)} - \fc r (2r-1)(r-1)  \right).
    \end{align*}
For the continuity of $E(\alpha,\cdot )$ and its derivative, we note that for $t = \alpha$, we have
\[
E(\alpha,\alpha^-)= \left(\fb +q-\fc r\right) \alpha^{2r-1} = 0
\]
and 
$$
\partial_t E(\alpha,\alpha^-)= (\fb+ q(2q-1)-\fc r (2r-1))\alpha^{2(r-1)}=2(q(q-1)-\fc r (r-1))\alpha^{2(r-1)}=0.
$$
Next, we remark that $\partial_t ^2 E(\alpha,\cdot)$ vanishes only once in $(0,\alpha)$. As $\partial_t E(\alpha,0)>0$ and $\partial_t E(\alpha,\alpha)=0$, then  $\partial_t E(\alpha,\cdot)$ is positive, then negative. It follows that $E(\alpha,\cdot)$ is increasing then decreasing, but since $E(\alpha, 0)= E(\alpha, \alpha) = 0$, $E$ is positive on $(0, \alpha)$. The bounds are now easy to obtain. 
\end{proof}

Now, we show that $\widetilde{u}_\mu$ converges, in $H^2(\RR^d)$, to $Q$, the unique positive solution to 
$$
-\Delta w+ w= \fc r w^{2r-1}. 
$$    
We use a fixed point procedure. We introduce the linearized operator $\cL_Q : W^{2,p}\cap H^2_{\rm rad} (\R^d) \to L^p \cap L^2_{\rm rad}(\R^d)$ defined by
    \[
        \cL_Q : v \mapsto - \Delta v + v - \fc r(2r-1) Q^{2r-2} v.
    \]
   As $Q$ is non-degenerate~\cite{Kwo-89}, then $\cL_Q$ is invertible for any $p\geq 2$.    
    For any $\alpha>0$, we write $\widetilde{u}_\alpha = Q + w_\alpha$, so that $\widetilde{u}_\alpha$ satisfies~\eqref{eq:IFT_NLS} iff $w_\alpha$ satisfies the fixed point equation $w_\alpha = G_\alpha(w_\alpha)$ with
    \[
        G_\alpha(w) := \cL_Q^{-1} \left\{ E(\alpha, Q + w) + \fc r \left[ (Q + w)^{2r-1} - Q^{2r-1} - (2r-1) Q^{2r-2} w  \right] \right\}.
    \]
   
    \begin{proposition}\label{prop:fixed_point_infinity} 
            There exists $\alpha_0>0$ such that, for all $\alpha\in (0,\alpha_0)$, there exists $\eta=\eta(\alpha)>0$
            so that the map $G_\alpha$ is a contraction on $\cB(\eta)$, where 
            \[
            \cB( \eta) := \left\{ w \in W^{2,d} \cap H^2_{\rm rad}(\R^d), \ \| w \|_{W^{2,d}} + \| w \|_{H^2} \le \eta \right\}.
            \] 
            In particular, the map $G_\alpha$ admits a unique fixed point in $\cB( \eta)$.
    \end{proposition}
Our choice of the space is such that $W^{2,d}(\RR^d)\subset C^{0,\gamma}(\RR^d)$, $\forall \gamma\in (0,1)$. 

\begin{proof}
    Let us first prove that for $\alpha$ and $\eta$ small enough, the map $G_\alpha$ leaves $\cB( \eta)$ invariant.  
  Since $\cL_Q^{-1}: L^p\cap L^2_{\rm rad} (\RR^d)\to W^{2,p}\cap H^2_{\rm rad} (\RR^d)$ is bounded, it is enough to control, for all $w\in \cB(0, \eta)$
  $$
 \norm{E(\alpha, Q + w) + \fc r \left[ (Q + w)^{2r-1} - Q^{2r-1} - (2r-1) Q^{2r-2} w
 	\right]   }_{L^p}$$
for $p\in\set{2,d}$.    
    For $w \in \cB(\eta)$, we have from Proposition~\ref{prop:E_alpha} that 
    \[
        \left\| E(\alpha, Q + w ) \right\|_p \le \bra{\fb+ q }\alpha^{2(r-1)} \| Q + w \|_p \le \bra{\fb+ q } \alpha^{2(r-1)} \left( \| Q \|_p + \eta \right).
    \]
  On the other hand, we have, with  $G(w) :=  |Q + w|^{2r-1} - Q^{2r-1} - (2r-1) Q^{{2r-2}} w$
    \[
  \av{G(w)}\le \begin{cases}
            C w^{2r-1} \quad & \text{for} \quad w \ge Q, \\
            C Q^{2r-3} w^2 = C (w/Q)^{-2r+1} w^{2r-1} \quad & \text{for} \quad | w | \le Q.
        \end{cases}
    \]
    If $2r-3\geq 0$, then 
    $$
\norm{   G(w)}_{p}\leq C \bra{\norm{Q}_\ii^{2r-3}+\norm{w}_\ii^{2r-3}}\norm{w}_p^2\leq C\bra{\norm{Q}_\ii^{2r-3}+\eta^{2r-3}}\eta^2. 
    $$
    If $2r-3<0$, when $\av{w}\leq Q$, we also have  $\av{G(w)} \leq C \av{w}^{2r-1}$, hence 
    \[
       \norm{G(w)}_p \le C  \eta^{2r - 1}.
    \]
    It follows that 
    \begin{equation}\label{eq:bound-on-G}
\norm{ G_\alpha(w)}_{W^{2,d}\cap H^2}\leq C \norm{\cL_Q^{-1}}\bra{ \alpha^{2r-2}\bra{\norm{Q}_p+\eta}+ \eta^2\bra{\norm{Q}_\ii^{2r-3}+ \eta^{2r-3}}}. 
    \end{equation}
    For any $\alpha$ small enough, we can thus choose $\eta$, which depends on $\alpha$,
    such that the R.H.S. of~\eqref{eq:bound-on-G} is smaller than $\eta$, thus $G_\alpha$ leaving the set $\cB(\eta)$ invariant.
    
    \medskip
    
    We now prove that $G_\alpha$ is a contraction. First, we have, using again Proposition~\ref{prop:E_alpha},
    \[
        \left| E(\alpha, w) - E(\alpha, v) \right| \le \| E'(\alpha; .) \|_\infty | w - v | \le C \alpha^{2(r-1)} | w - v |,
    \]
    so
    \[
        \left\| E(\alpha, w) - E(\alpha, v) \right\|_p \le C \alpha^{2(r-1)} \| w - v \|_p.
    \]
    Next, we have 
    \begin{align*}
      &   \left|G(w) - G(v) \right| \le (2r-1) \left( \int_0^1 \left| (Q + tw + (1 - t)v )^{2r-2} - Q^{2r-2} \right| \rd t \right)  | w - v |.
    \end{align*}
    Setting $z = tw + (1 - t)v$, we have
    \[
\av{    |Q + z|^{2r-2} - Q^{2r-2}} \le \begin{cases}
        C z^{2r-2} \quad & \text{for} \quad z \ge Q, \\
        C (2r-2) Q^{2r-3} z = C (2r-2) (z/Q)^{3-2r} z^{2r-2} \quad & \text{for} \quad | z | \le Q.
    \end{cases}
    \]
    Again, in both cases, we get an upper bound of the form 
     \begin{align*}
         \left\| G(w) - G(v) \right\|_p  \le C \left( \norm{Q}_\ii^{2r-3} + \eta^{2r-3} \right) \eta  \norm{ w - v }_p.
    \end{align*}
    Altogether, we have proved that
    \[
        \left\| G_\alpha(w) - G_\alpha(v) \right\|_{H^2 \cap W^{2,d}} \le C \left( \alpha^{2(r-1)} +  \eta ^{2(r-1)} + \norm{Q}_\ii \eta \right) \| w - v \|_{H^2\cap W^{2,d}}.
    \]
    So, for $\eta$ and $\alpha$ small enough, we can apply the fixed point theorem, and deduce that $G_\alpha$ has a unique fixed point in $\cB(\eta)$. 
\end{proof}
\begin{remark}\label{rem:convergence_fixedpoint}
    In the proof of Proposition~\ref{prop:fixed_point_infinity}, given $\alpha$ small enough, $\eta$ has to be choosen such that 
    \begin{align*}
        C \norm{\cL_Q^{-1}}\bra{ \alpha^{2r-2}\bra{\norm{Q}_p+\eta}+ \eta^2\bra{\norm{Q}_\ii^{2r-3}+ \eta^{2r-3}}}\le \eta\\
        C \left( \alpha^{2(r-1)} +  \eta ^{2(r-1)} + \norm{Q}_\ii \eta \right)<\lambda
    \end{align*}
    for some $\lambda\in (0,1)$. In particular, one can choose $\eta=\alpha^{\beta}$ with $0<\beta<2(r-1)$. 
\end{remark}

Since $\widetilde{u}_\mu$ is the (unique) ground state of~\eqref{eq:IFT_NLS} with $\alpha=\alpha_\mu$, we deduce that $\widetilde{u}_\mu=Q+w_{\alpha_\mu}$ with $w_{\alpha_\mu}$ the unique fixed point of $G_{\alpha_\mu}$. As a consequence, 
\[
   \left\| \widetilde{u}_\mu -  Q \right\|_{H^2}= \left\| w_{\alpha_\mu} \right\|_{H^2} =  \left\| G_{\alpha_\mu}(w_{\alpha_\mu}) \right\|_{H^2}\le \alpha_\mu^{\beta}=\left(\frac{1}{\fb +\mu}\right)^{\frac{\beta}{2(r-1)}} \xrightarrow[\mu \to \infty]{} 0.
\]

This concludes the proof of Lemma~\ref{lem:asymp_infinity}.


\subsection*{Conflict of interest}
The authors declare no conflict of interest.

\printbibliography

\end{document}